\documentclass[fleqn,3p, times]{elsarticle}
\biboptions{sort&compress}

\usepackage[utf8]{inputenc}
\usepackage{amsmath}
\usepackage{amssymb}
\usepackage{amsthm}
\usepackage{hyperref}
\usepackage{bm}
\usepackage{enumitem}
\usepackage{overpic}
\usepackage{algorithm}
\usepackage{algpseudocode}
\usepackage{microtype}
\usepackage{subfig}

\newcommand{\new}[1]{#1} 

\renewcommand{\=}{\,=\,}

\renewcommand{\vec}{\bm}
\newcommand{\tensor}{\bm}

\newcommand{\p}{p_\mathrm{w}}
\newcommand{\q}{\vec{q}_\mathrm{w}}
\renewcommand{\u}{\vec{u}}
\newcommand{\pzero}{p_{\mathrm{w},0}}

\newcommand{\uzero}{\vec{u}_0}
\newcommand{\szero}{s_\mathrm{w,0}}

\newcommand{\pquh}[1]{(p,\vec{q},\vec{u})_h^{#1}}
\newcommand{\dpquh}[1]{\varDelta \pquh{#1}}
\newcommand{\incph}{\varDelta p_h}
\newcommand{\incqh}{\varDelta \vec{q}_h}
\newcommand{\incuh}{\varDelta \vec{u}_h}

\newcommand{\pqh}[1]{(p,\vec{q})_h^{#1}}
\newcommand{\dpqh}[1]{\varDelta \pqh{#1}}

\newcommand{\xn}[1]{\mathbf{x}^{#1}}

\newcommand{\plainFP}{\mathcal{FP}}
\newcommand{\FP}[1]{\plainFP(#1)}
\newcommand{\dFP}[1]{\Delta\FP{#1}}

\newcommand{\s}{s_\mathrm{w}}
\newcommand{\smin}{s_{\mathrm{w},\mathrm{min}}}
\newcommand{\permeability}{k_\mathrm{w}}
\newcommand{\rhob}{\rho_\mathrm{b}}
\newcommand{\rhow}{\rho_\mathrm{w}}
\newcommand{\ppore}{p_\mathrm{E}}
\newcommand{\pc}{p_\mathrm{c}}
\newcommand{\pw}{p_\mathrm{w}}

\newcommand{\poroelasticstress}{\tensor{\sigma}^\mathrm{por}}
\newcommand{\stress}{\tensor{\sigma}}
\newcommand{\I}{\tensor{I}}
\newcommand{\g}{\vec{g}}

\renewcommand{\k}{k_\mathrm{abs}}

\newcommand{\n}{\vec{n}}

\newcommand{\NABLA}{\vec{\nabla}}
\newcommand{\DIV}{\NABLA\cdot}
\newcommand{\GRAD}{\NABLA}
\newcommand{\eps}[1]{\tensor{\varepsilon}(#1)}

\newcommand{\Vh}{\bm{V}_h}
\newcommand{\Qh}{W_h}
\newcommand{\Zh}{\bm{Z}_h}
\newcommand{\QZVh}{\Qh \times \Zh \times \Vh}
\newcommand{\QZh}{\Qh \times \Zh}

\newcommand{\ph}{p_h}
\newcommand{\qh}{\vec{q}_h}
\newcommand{\uh}{\vec{u}_h}
\newcommand{\testph}{w_h}
\newcommand{\testqh}{\vec{z}_h}
\newcommand{\testuh}{\vec{v}_h}
\newcommand{\testpqh}{(w,\vec{z})_h}
\newcommand{\testpquh}{(w,\vec{z},\vec{v})_h}

\newcommand{\sn}[1]{\s^{n,#1}} 
\newcommand{\pn}[1]{\ph^{n, #1}}
\newcommand{\qn}[1]{\qh^{n, #1}}
\newcommand{\un}[1]{\uh^{n, #1}}

\newcommand{\dpn}[1]{\varDelta \ph^{n,#1}}
\newcommand{\dqn}[1]{\varDelta \qh^{n,#1}}
\newcommand{\dun}[1]{\varDelta \uh^{n,#1}}

\newcommand{\dpqun}[1]{\varDelta (p,\vec{q},\vec{u})_h^{n,#1}}

\newcommand{\aVG}{{a_\mathrm{vG}}}
\newcommand{\nVG}{{n_\mathrm{vG}}}

\newcommand{\A}{\mathbf{A}}
\newcommand{\Id}{\mathbf{I}}
\renewcommand{\b}{\mathbf{b}}
\newcommand{\x}[1]{\mathbf{x}^{#1}}
\newcommand{\e}[1]{\mathbf{e}^{#1}}
\newcommand{\hate}[1]{\hat{\mathbf{e}}^{#1}}
\newcommand{\betaWeight}{\bm{\beta}}

\newcommand{\bV}{\mathbf{b}}
\newcommand{\uV}[1]{\mathbf{u}^{#1}}
\newcommand{\pV}[1]{\mathbf{p}^{#1}}
\newcommand{\qV}[1]{\mathbf{q}^{#1}}
\newcommand{\hV}{\mathbf{h}}
\newcommand{\nullV}{\mathbf{0}}
\newcommand{\peV}{\mathbf{p}_\mathrm{E}}
\newcommand{\phiV}{\bm{\phi}_\mathrm{0}}
\newcommand{\phipV}{\bm{\phi}}
\newcommand{\pVSpace}{\mathbf{P}_{\phi\geq0}}
\newcommand{\DM}{\mathbf{D}}
\newcommand{\KM}{\mathbf{K}}
\newcommand{\km}{k_\mathrm{m}}
\newcommand{\kM}{k_\mathrm{M}}
\newcommand{\epV}[1]{\mathbf{e}_{\pV{}}^{#1}}

\newcommand{\NP}{n_\mathrm{p}}
\newcommand{\NQ}{n_\mathrm{q}}
\newcommand{\NU}{n_\mathrm{u}}

\newcommand{\LipschitzB}{L_\mathrm{b}}
\newcommand{\lipschitzB}{l_\mathrm{b}}
\newcommand{\LipschitzK}{L_\mathrm{K}}
\newcommand{\LipschitzS}{L_\mathrm{s}}
\newcommand{\LipschitzLambdaT}{\tilde{L}_{\mathrm{K}}}

\newcommand{\DbV}{\mathbf{D}_\mathbf{b}}

\newcommand{\StiffnessMechanics}{\mathbf{A}_\mathrm{uu}}
\newcommand{\DivPU}{\mathbf{D}_\mathrm{pu}}
\newcommand{\DivPQ}{\mathbf{D}_\mathrm{pq}}
\newcommand{\MassMatrixP}[1]{\mathbf{M}_\mathrm{pp}^{#1}}
\newcommand{\MassMatrixQ}[1]{\mathbf{M}_\mathrm{qq}^{#1}}
\newcommand{\LMatrix}{\mathbf{L}_{\mathrm{pp}}}
\newcommand{\MassSaturation}{\mathbf{S}_{\mathrm{pp}}}
\newcommand{\MassPermeability}{\mathbf{K}_\mathrm{qq}}
\newcommand{\rhsVMechanics}{\mathbf{f}_\mathrm{u}}
\newcommand{\rhsVPressure}{\mathbf{f}_\mathrm{p}}
\newcommand{\rhsVFlux}{\mathbf{f}_\mathrm{q}}

\newtheorem{theorem}{Theorem}

\newtheorem{lemma}[theorem]{Lemma}
\newtheorem{corollary}[theorem]{Corollary}

\newtheorem{remark}{Remark}

\theoremstyle{definition}

\newtheorem{algo}{Algorithm}

\newenvironment{indentblock}{\noindent \hspace{0.5cm} \begin{minipage}{0.9\textwidth}}{\end{minipage}}

\begin{document}

\begin{frontmatter}
 \title{Anderson accelerated fixed-stress splitting schemes for consolidation of unsaturated porous media}

 \author[uib]{Jakub~Wiktor~Both\corref{cor1}}
 \ead{jakub.both@uib.no}
 \author[uib]{Kundan~Kumar}
 \ead{kundan.kumar@uib.no}
 \author[uib,princeton]{Jan~Martin~Nordbotten}
 \ead{jan.nordbotten@uib.no}
 \author[uib]{Florin~Adrian~Radu}
 \ead{florin.radu@uib.no}
 
 \address[uib]{Department of Mathematics, University of Bergen, Bergen, Norway}
 \address[princeton]{Department of Civil and Environmental Engineering, Princeton University, Princeton, NJ, USA}
 
 \cortext[cor1]{Corresponding author}
 
\begin{abstract}

 In this paper, we study the robust linearization of nonlinear poromechanics of unsaturated materials.
 The model of interest couples the Richards equation with linear elasticity equations, employing the 
 equivalent pore pressure. 
 In practice a monolithic solver is not always available, defining the requirement for a linearization scheme
 to allow the use of separate simulators, which is not met by the classical Newton method.
 We propose three different linearization schemes incorporating the fixed-stress splitting scheme,
 coupled with an L-scheme, Modified Picard and Newton linearization of the flow. 
 All schemes allow the efficient and robust decoupling of mechanics and flow equations. 
 In particular, the simplest scheme, the Fixed-Stress-L-scheme, employs solely constant diagonal stabilization,
 has low cost per iteration, and is very robust.
 Under mild, physical assumptions, it is theoretically shown to be a contraction. 
 Due to possible break-down or slow convergence of all considered splitting schemes, Anderson acceleration 
 is applied as post-processing. 
 Based on a special case, we justify theoretically the general ability of the Anderson acceleration 
 to effectively accelerate convergence and stabilize the underlying scheme, allowing even non-contractive 
 fixed-point iterations to converge.
 To our knowledge, this is the first theoretical indication of this kind. 
 Theoretical findings are confirmed by numerical results. 
 In particular, Anderson acceleration has been demonstrated to be very effective for the considered
 Picard-type methods.
 Finally, the Fixed-Stress-Newton scheme combined with Anderson acceleration provides a robust 
 linearization scheme, meeting the above criteria.

\end{abstract}

  
\end{frontmatter}

\section{Introduction}

The coupling of fluid flow and mechanical deformation in unsaturated porous media is relevant
for many applications ranging from modeling rainfall-induced land subsidence or levee failure
to understanding the swelling and drying-shrinkage of wooden or cement-based materials. 
Assuming linear elastic behavior, the process can be modeled by coupling the Richards 
equations with quasi-static linear elasticity equations, generalizing the classical Biot 
equations~\cite{Biot1941}. In this work, we consider utilize the equivalent pore pressure~\cite{Coussy2004},
which allows a thermodynamically stable formulation~\cite{Kim2013}.

For the numerical simulation of large scale applications, the solution of linear problems is typically
the computationally most expensive component. For nonlinear problems, the dominating 
cost is determined by both the chosen linearization scheme and solver technology. In particular,
the choice of a linearization scheme defines the requirements for the solver technology. 
Commonly, Newton's method is the first choice linearization scheme. However, for the nonlinear Biot 
equations, as monolithic solver, Newton's method requires the solver technology to solve saddle point problems coupling
mechanics and flow equations. Additionally, in practice, constitutive laws employed in the model might be 
not Lipschitz continuous~\cite{vanGenuchten1980}. Thus, the arising systems are ill-conditioned
and require an advanced, monolithic simulator. As the latter might be not available, the goal of
this work is to develop a linearization scheme, which is robust and allows the use of decoupled 
simulators for mechanics and flow equations. For this purpose, we adopt closely related concepts 
for the linear Biot equations and the Richards equations.

For the numerical solution of the linear Biot equations, splitting schemes are widely used; either
as iterative solvers~\cite{Settari1998} or as preconditioners~\cite{White2011}. In particular, the 
fixed-stress splitting scheme has aroused much interest, being unconditionally stable in the sense
of a von Neumann analysis~\cite{Kim2009} and a global contraction~\cite{Mikelic2013,Both2017,Bause2017}. 
\new{As iterative solver, the scheme has been extended in various ways; e.g., it can be rewritten to a 
parallel-in-time solver~\cite{Borregales2018}, a multi-scale version allowing separate grids
for the mechanics and flow problem has been developed~\cite{Dana2018}, and the concept has been
extended to nonlinear multi-phase flow coupled with linear elasticity~\cite{Kim2013}.}
In the context of monolithic solvers, the scheme has been applied as preconditioner for Krylov 
subspace methods~\cite{Castelleto2015,Castelleto2016,White2016,Adler2017} and as smoother for 
multigrid methods~\cite{Gaspar2017}. All in all, the scheme defines a promising
strategy to decouple mechanics and flow equations.

\new{For the linearization of the Richards equation, the standard Newton method has to be used with care,
since the Richards equation is a degenerate elliptic-parabolic equation, modeling saturated/unsaturated flow, 
and additionally material laws might be H\"older continuous. Various problem-specific alternatives have been developed 
in the literature. We want to point out two particular, simple linearization schemes; the L-scheme and the Modified Picard method.}
The L-scheme~\cite{Slodicka2002}, employs diagonal stabilization for monotone, Lipschitz
continuous nonlinearities. Global convergence has been rigorously proven for several porous media 
applications~\cite{Pop2004,Radu2014,Radu2017}; in particular also for the Richards equation~\cite{List2016}.
\new{The L-scheme can be also applied for H\"older continuous problems~\cite{Radu2017,Both2018b}.
Furthermore, for the Richards equations it can be used to define a robust, linear domain decomposition method~\cite{Seus2018}.}
The L-scheme linearization has been coupled with the fixed-stress splitting scheme for nonlinear Biot's equations with linear
coupling~\cite{Borregales2017}.  Less robust, but in some cases more efficient is the Modified Picard
method~\cite{Celia1990}, which employs the first order Taylor approximation for the saturation, 
still allowing for a H\"older continuous permeability.

In this paper, we combine the fixed-stress splitting scheme with the L-scheme, the Modified Picard 
method and Newton's method. The resulting schemes decouple and linearize simultaneously the mechanics
and flow equations, utilizing only a single loop and allowing for separate simulators. We show theoretically linear convergence of 
the Fixed-Stress-L-scheme, assuming non-vanishing residual saturation, permeability and porosity. 
However, the theoretical convergence rate might deteriorate in unfavorable situations, leading to 
either slow convergence or even stagnation in practice. As remedy we apply Anderson acceleration.

Anderson acceleration, originally introduced by~\cite{Anderson1965}, in order to accelerate fixed 
point iterations in electronic structure computation, has been successfully applied in various 
other fields; in particular for the modified Picard iteration~\cite{Lott2012}. Reusing previous 
iterations to approximate directional derivatives it can be related to a multi-secant quasi-Newton
method~\cite{Fang2009} and to preconditioned GMRES for linear problems~\cite{Walker2011}. For 
nonlinear problems, it naturally generalizes to a preconditioned nonlinear GMRES. Being a 
post-processing, it can be combined with splitting methods, still allowing separate simulators 
unlike preconditioned monolithic solvers. So far, theoretical results guarantee only convergence 
for contractive fixed-point iterations~\cite{Toth2015}. Furthermore, those results do not guarantee 
actual acceleration. Based on a special case, we justify theoretically the ability of the Anderson
acceleration to effectively accelerate convergence of contractive fixed-point iterations and 
moreover stabilize non-contractive fixed-point iterations. Applying Anderson acceleration 
for diverging methods might allow convergence after all. To our knowledge, this is the first 
theoretical indication of this kind.  
Other stabilization techniques could be applied as adaptive step size control, 
adaptive time stepping or the combination of a Picard-type method with a Newton-type method, 
following ideas by~\cite{Paniconi1994}. These concepts have not been considered in the scope 
of this work.

We present numerical results confirming the theoretical findings of this work. Indeed, the 
Fixed-Stress-L-scheme is more robust than the modifications employing Newton's method and 
the Modified Picard method. Moreover, convergence of the Picard-type methods can be 
significantly accelerated by the Anderson acceleration. When applied to initially diverging
methods, convergence can be reliably retained.

The main, new contributions of this work are:
\begin{itemize}
 
 \item We propose three linearization schemes incorporating the fixed-stress splitting scheme, 
 coupled with an L-scheme, Modified Picard and Newton linearization of the flow. All schemes 
 allow the efficient and robust decoupling of mechanics and flow equations. For the simplest 
 scheme, the Fixed-Stress-L-scheme, we show theoretical convergence, assuming non-vanishing 
 residual saturation, permeability and porosity, cf.\ Theorem~\ref{theorem:convergence:fsl}.

 \item Based on a special case, we justify theoretically the general ability of the Anderson 
 acceleration to effectively accelerate convergence and stabilize the underlying scheme, allowing
 even non-contractive fixed-point iterations to converge, cf.\ Section~\ref{section:linear-analysis:restarted-aa:discussion}.

 \item The combination of the proposed linearization schemes and Anderson acceleration is 
 demonstrated numerically to be robust and efficient. In particular, Anderson acceleration 
 allows the schemes to converge in challenging situations. The Fixed-Stress-Newton method 
 coupled with Anderson acceleration shows best performance among the splitting schemes,
 cf.\ Section~\ref{section:numerical-results}.
  
\end{itemize}

\noindent
The paper is organized as follows. In Section~\ref{section:model}, the mathematical model for the nonlinear
Biot equations is explained. In Section~\ref{section:discretization}, a three-field discretization is introduced,
employing linear Galerkin finite elements and mixed finite element for the mechanics and flow equations, respectively.
In Section~\ref{section:linearization}, we recall the monolithic Newton method and 
introduce three splitting schemes,
simultaneously linearizing and decoupling the mechanics and flow equations. 
In Section~\ref{section:fsl:theory}, convergence is proved for the Fixed-Stress-L-scheme. 
In Section~\ref{section:acceleration}, Anderson acceleration is recalled, and in 
Section~\ref{section:linear-analysis:restarted-aa}, the ability of the Anderson acceleration 
to effectively accelerate convergence and increase robustness is discussed theoretically.
In Section~\ref{section:numerical-results}, numerical results are presented, 
illustrating in particular the increase of robustness via Anderson acceleration.
The work is closed with concluding remarks in Section~\ref{section:concluding-remarks}.

\section{Mathematical model -- Nonlinear Biot's equations coupling Richards equation and linear elasticity}\label{section:model}
We consider a nonlinear extension of the classical, linear Biot equations modeling
flow in deformable porous media under possibly both fully and partially saturated conditions.
Further more we assume:

\begin{itemize}

 \item[(A1)] The bulk material is linearly elastic and deforms solely under infinitesimal deformations.
 
 \item[(A2)] There exists two fluid phases -- one active and one passive phase (standard assumption
 for the Richards equation).
 
 \item[(A3)] The active fluid phase is incompressible and corresponding fluxes are described by
 Darcy's law.
 
 \item[(A4)] Mechanical inertia effects are negligible allowing to consider the quasi-static 
 balance of linear momentum.
 
\end{itemize}

\noindent
We model the medium at initial conditions by a reference 
configuration $\Omega\subset\mathbb{R}^{d}$, $d\in\{2,3\}$. Due to the limitation 
to infinitesimal deformations, the domain of the primary fields is approximated by 
$\Omega$ on the entire time interval of interest $(0,T)$, with final time $T>0$.

Finally, the governing equations describing coupled fluid flow and mechanical deformation of a 
porous medium with mechanical displacement~$\u$, fluid pressure~$\p$ and volumetric flux~$\q$ as primary 
variables are given by
\begin{align}
  \partial_t \left( \phi \s \right) + \DIV \q &\= 0, \label{model:p}\\
  \q + \permeability( \s ) \left( \GRAD \p - \rhow \g \right) &\= \bm{0}, \label{model:q}\\
  -\DIV \left[2\mu \eps{\u} + \lambda \DIV \u \I  - \alpha \ppore(\p) \I \right] &\= \rhob \g, \label{model:u}
\end{align}
where $\phi$ denotes variable porosity, $\s$ denotes fluid saturation,
$\permeability$ denotes fluid-dependent mobility, $\rhow$ and $\rhob$ denote fluid and bulk density,
respectively, $\g$ is the gravitational acceleration,
$\eps{\u}$ and $\DIV \u$ denote the linear strain and the volumetric deformation, respectively,
$\mu$ and $\lambda$ denote the Lam\'e parameters, $\alpha$ is the Biot coefficient
and $\ppore$ denotes the pore pressure.
In the following, we comment briefly on the single components of the mathematical model and 
refer to~\cite{Coussy2004} for a detailed derivation.

\begin{itemize}

 \item[Eq.~\eqref{model:p}:] 
 For the fluid flow, an active and a passive fluid phase are assumed. In other words,
 the passive phase responds instantaneously to the active phase and therefore has a 
 constant pressure. The behavior of the active fluid phase is governed by mass conservation, identical
 to volume conservation for an incompressible fluid. The volume is given by the product
 of porosity $\phi$ and saturation $\s$. The porosity changes linearly with volumetric 
 deformation $\DIV \u$ and pore pressure $\ppore$ by
 \begin{align} \label{model:porosity}
  \phi(\u,\p) \= \phi_0 + \alpha \DIV (\u - \uzero) + \frac{1}{N} (\ppore(\p) - \ppore(\pzero)),
 \end{align}
 where $\phi_0$, $\uzero$ and $\pzero$ are the initial porosity, displacement and pressure, respectively,
 and $\alpha$ is the Biot coefficient and $N$ is the Biot modulus. Eq.~\eqref{model:porosity} is a byproduct of 
 the thermodynamic derivation of the effective stress by Coussy~\cite{Coussy2004}.
 Furthermore, the saturation $\s=\s(p)$ is assumed to be described by a material law 
 $\s:\mathbb{R}\rightarrow(0,1]$, satisfying $\s(p)=1$, $p\geq0$, and having a negative 
 inverse $\pc : (0,1]\rightarrow \mathbb{R}_{+}$, satisfying $\s(-\pc(s))=s$, $s\in(0,1]$. 
 In the literature, the function $\pc$ is often referred to as capillary pressure.

 \item[Eq.~\eqref{model:q}] The volumetric flux $\q$ is assumed to be described by Darcy's
 law for multiphase flow. Here, the permeability scaled by viscosity is 
 given by a material law $\permeability=\permeability(\s)$. 
 In practice, the material laws can become H\"older continuous.

 \item[Eq.~\eqref{model:u}:]
 The mechanical behavior is governed by balance of linear momentum under quasi-static 
 conditions, combined with an effective stress formulation.  Allowing only for small 
 deformations, we employ the St.Venant Kirchhoff model for the effective stress,
 determining the poroelastic stress as
 $\poroelasticstress(\u,\pw) = 2\mu \eps{\u} + \lambda \DIV \u \I - \alpha \ppore(\pw)$.
 As pore pressure, we use the equivalent pore pressure~\cite{Coussy2004}
 \begin{align}\label{model:pore-pressure:definition}
  \ppore(p) \= \s(p) p - \int_{\s(p)}^1 \pc(s) \, ds,
 \end{align}
 which takes into account interfacial effects. By construction it satisfies $d\ppore = \s(p)\,dp$.
 As body force we assume solely gravity, where for the sake of simplicity the bulk density 
 $\rhob$ is assumed to be constant.
 All in all, Eq.~\eqref{model:u} acts as compatibility condition to be satisfied at each time.
  
\end{itemize}

\noindent
Introducing two partitions $\Gamma_D^f\cup\Gamma_N^f = \Gamma_D^m\cup\Gamma_N^m = \partial \Omega$
of the boundary of $\Omega$ and the outer normal $\n$ on $\partial\Omega$, we assume boundary 
conditions and initial conditions 
\begin{align*}
  \p&=p_{w,D} & &\text{on }\Gamma_D^f\times(0,T), 			& \u &= \u_D & & \text{on }\Gamma_D^m\times(0,T), \\
  \q\cdot \n &= \vec{q}_{w,N} & & \text{on }\Gamma_N^f\times(0,T),      & \poroelasticstress(\u,\pw)\n &= \poroelasticstress_n & & \text{on }\Gamma_N^m\times(0,T), \\
  \p&=\pzero & &\text{in }\Omega\times\{0\}, 			        & \u& =\uzero & &\text{in }\Omega\times\{0\}.
\end{align*}

\noindent
All in all, the nonlinear Biot equations~\eqref{model:p}--\eqref{model:u} couple nonlinearly 
the Richards equation and linear elasticity equations. In the fully saturated 
regime ($\p\geq0$), the model reduces locally to the classical, linear Biot equations for 
an incompressible fluid and compressible rock. We note, as long as the fluid saturation is not vanishing,
the nonlinear Biot equations~\eqref{model:p}--\eqref{model:u} are parabolic, unlike the 
degenerate elliptic-parabolic Richards equation, cf.\ Remark~\ref{remark:model:properties}.

\section{Finite element discretization}\label{section:discretization}

We discretize the Biot equations~\eqref{model:p}--\eqref{model:u} in space and time 
by the finite element method and the implicit Euler method, respectively.
More precisely, given a regular triangulation $\mathcal{T}_h$ of the domain $\Omega$,
we employ linear, constant and lowest order Raviart-Thomas finite elements to approximate displacement, 
pressure and volumetric flux, respectively. For the sake of simplicity, we assume zero boundary 
conditions on $\partial\Omega=\Gamma_D^m=\Gamma_D^f$. The corresponding discrete function spaces are then given by
\begin{align*}
  \Qh &\= \Bigl\{ \testph \in L^2(\Omega) 		 \, \Bigm| \, \forall T \in \mathcal{T}_h, \testph|_T \in {{\mathbb{P}}_0}       \Bigr\}, \\
  \Zh &\= \Bigl\{ \testqh \in {H(\mathrm{div}; \Omega)}  \, \Bigm| \, \forall T \in \mathcal{T}_h, \testqh|_T (\vec{x})=\vec{a}+b \vec{x},\ \vec{a} \in  \mathbb{R}^d, b \in \mathbb{R} \Bigr\}, \\
  \Vh &\= \Bigl\{ \testuh \in {[H^1_0(\Omega)]}^d 	 \, \Bigm| \, \forall T \in \mathcal{T}_h, \testuh|_T \in {[{\mathbb{P}}_1]}^d   \Bigr\},
\end{align*}
where ${\mathbb{P}}_0$ and ${\mathbb{P}}_1$ denote the spaces of scalar piecewise constant and
piecewise (bi-)linear functions, respectively, and elements of $\Vh$ are zero on the boundary.
\new{We note, that the chosen discretization is not stable with respect to the full range of material parameters~\cite{Rodrigo2017}, e.g., for very small permeability.
However, the specific choice of the finite element spaces is not essential for the further}
discussion of the linearization. Furthermore, for the temporal discretization we employ a 
partition $\{t^n\}_n$ of the time interval $(0,T)$ with (constant) time step size $\tau= t^n - t^{n-1} > 0$.

Then given initial data $(p,\vec{u})_h^0\in\Qh\times\Vh$, at each time step $n\geq1$, the discrete problem reads: 
Given $\pquh{n-1}\in\QZVh$, find $\pquh{n}\in\QZVh$, satisfying for all $\testpquh\in \QZVh$
\begin{align}
 \langle \phi^{n-1} (\s^n  - \s^{n-1}), \testph \rangle + \alpha \langle \s^n \DIV (\uh^n - \uh^{n-1}), \testph \rangle
+ \tfrac{1}{N} \langle\s^n ( \ppore^n - \ppore^{n-1}) , \testph \rangle  + \tau \langle \DIV\qh^n, \testph \rangle &\= 0,                                                                  \label{discretization:p}\\
 \langle \permeability(\s^n)^{-1} \qh^n, \testqh \rangle  - \langle \ph^n, \DIV \testqh \rangle &\= \langle \rhow \g, \testqh \rangle,                                                  \label{discretization:q} \\
 2\mu\langle \eps{\uh^n}, \eps{\testuh} \rangle + \lambda \langle \DIV \uh^n, \DIV \testuh \rangle - \alpha \langle \ppore^n, \DIV \testuh \rangle &\= \langle \rhob \g, \testuh \rangle,     \label{discretization:u} 
\end{align}
where $\s^k = \s(\ph^k)$ and $\ppore^k=\ppore(\ph^k)$, $k\in\{n-1,n\}$ and $\phi^{n-1}=\phi(\uh^{n-1},\ph^{n-1})$.
Here, $\langle \cdot, \cdot \rangle$ denotes the standard $L^2(\Omega)$ scalar product.

\begin{remark}[Volume conservation]
 The discretization~\eqref{discretization:p}--\eqref{discretization:u} is volume-conservative as 
 by Eq.~\eqref{model:porosity} it holds
 \begin{align*}
  \phi^n\s^n - \phi^{n-1}\s^{n-1}
  &\= \phi^{n-1} (\s^n -\s^{n-1}) +\s^n \left( \alpha \DIV (\u^n - \u^{n-1}) + \frac{1}{N} (\ppore^n - \ppore^{n-1}) \right).
 \end{align*}
\end{remark}


\section{Monolithic and decoupled linearization schemes}\label{section:linearization}

In the following, we consider four linearization schemes. First, we apply the monolithic 
Newton method, being commonly the first choice when linearizing a nonlinear problem. 
Second, we propose a linearization scheme, which employs constant diagonal stabilization 
in order to linearize and decouple simultaneously the mechanics and flow equations. 
Furthermore, we introduce two modifications of the latter method, utilizing both 
decoupling and first order Taylor approximations. For direct comparison, we formulate 
all schemes in incremental form.

\paragraph{Monolithic schemes vs.\ iterative operator splitting schemes for saddle point problems}
Biot's equations yield a saddle point problem, which in general are difficult to solve.
Containing more information on coupling terms, a monolithic scheme for this purpose 
is per se more stable, whereas for iterative operator splitting schemes, stability 
is always an issue necessary to be checked. However, in contrast to robust 
splitting schemes, for a monolithic scheme a fully coupled simulator with advanced solver 
technology is required. For that, one possibility is to apply a splitting scheme
as either an iterative solver or a preconditioner for a Krylov subspace method.
The latter is more efficient and robust, cf.~e.g.~\cite{Castelleto2015} for the 
linear Biot equations. In case only separate simulators are available, the concept 
of preconditioning the coupled problem cannot be applied in the same sense. But we note that 
acceleration techniques as Anderson acceleration can be applied as post-processing to the 
iterative splitting schemes, acting as preconditioned nonlinear GMRES solvers applied 
to the coupled problem, cf.\ Section~\ref{section:acceleration}.

\subsection{Notation of residuals}

For the incremental formulation of the linearization schemes, we introduce naturally 
defined residuals of the coupled problem~\eqref{model:p}--\eqref{model:u}.
Given data $\pquh{n-1}\in\QZVh$ for time step $n-1$, the residuals at time step $n$ evaluated 
at some state $\pquh{}\in\QZVh$ and tested with $\testpquh\in\QZVh$
are defined by
\begin{align*}
 r_p^n(\pquh{};\testph) &= - \Big( \langle \phi^{n-1} (\s(\ph)  - \s^{n-1}), \testph \rangle + \alpha \langle \s(\ph) \DIV (\uh - \uh^{n-1}), \testph \rangle  \\
 &\quad\qquad + \frac{1}{N} \langle \s(\ph) ( \ppore(\ph) - \ppore^{n-1}), \testph \rangle + \tau \langle \DIV\qh, \testph \rangle \Big), \\
 r_q^n(\pquh{};\testqh) &= \langle \rhow \g, \testqh \rangle - \Big(\langle \permeability(\s(\ph))^{-1} \qh, \testqh \rangle  - \langle \ph, \DIV \testqh \rangle\Big), \\
 r_u^n(\pquh{};\testuh) &= \langle \rhob \g, \testuh \rangle - \Big( 2\mu\langle \eps{\uh}, \eps{\testuh} \rangle + \lambda \langle \DIV \uh, \DIV \testuh \rangle - \alpha \langle \ppore(\ph) , \DIV \testuh \rangle \Big).
\end{align*}
Shorter, given a sequence of approximations $\pquh{n,i}\in\QZVh$, $i\in\mathbb{N}$, of $\pquh{n}\in\QZVh$, we define
\begin{align*}
 r_p^{n,i}(\testph) &\= r_p^n(\pquh{n,i};\testph), \\
 r_q^{n,i}(\testqh) &\= r_q^n(\pquh{n,i};\testqh), \\
 r_u^{n,i}(\testuh) &\= r_u^n(\pquh{n,i};\testuh), \\
 r_u^{n,i/i-1}(\testuh) &\= r_u^n(\pqh{n,i},\un{i-1};\testuh).
\end{align*}

\subsection{Monolithic Newton's method}

We apply the standard Newton method linearizing the coupled problem~\eqref{discretization:p}--\eqref{discretization:u} 
in a monolithic fashion.

\paragraph{Scheme}
The monolithic Newton method reads: 
For each time step, given the initial guess $\pquh{n,0}=\pquh{n-1}$, loop over the iterations $i\in\mathbb{N}$ until convergence is reached. 
Given data at the previous time step $n-1$ and iteration $i-1$, find the increments
$\dpqun{i}\in\QZVh$, satisfying the coupled, linear problem, for all $\testpquh\in\QZVh$,
\begin{align}
  \left\langle \left(\phi^{i-1} \frac{\partial \s}{\partial \pw}(\pn{i-1}) + \frac{1}{N} (\sn{i-1})^2\right) \dpn{i}, \testph \right\rangle 
  + \alpha \langle \sn{i-1} \DIV \dun{i}, \testph \rangle
  + \tau \langle \DIV \dqn{i}, \testph \rangle &\= r_p^{n,i-1}(\testph),     								  \label{newton:p} \\
  \langle \permeability(\sn{i-1})^{-1} \dqn{i}, \testqh \rangle  
  + \left\langle \left(\left.\frac{\partial}{\partial \pw}\permeability(\sn{i-1})\right|_{\pn{i-1}}\right)^{-1} \qn{i-1} \dpn{i}, \testqh \right\rangle 
  - \langle \dpn{i}, \DIV \testqh \rangle &\= r_q^{n,i-1}(\testqh),         								  \label{newton:q} \\
  2\mu\langle \eps{\dun{i}}, \eps{\testuh} \rangle + \lambda \langle \DIV \dun{i}, \DIV \testuh \rangle
    - \alpha \langle \sn{i-1} \dpn{i}, \DIV \testuh \rangle &\= r_u^{n,i-1}(\testuh),                                                   \label{newton:u} 
\end{align}
and set
\begin{align*}
  \pquh{n,i} \= \pquh{n,i-1} + \dpquh{n,i}.
\end{align*}
After convergence is reached at iteration $N$, set $\pquh{n}=\pquh{n,N}$.

\paragraph{Properties}
Newton's method is known to be quadratically, locally convergent, which makes the method commonly the first choice linearization
method. However, in general it is not robust and has the following drawbacks:

\begin{itemize}
 
 \item In order to ensure convergence, the time step size has to be chosen sufficiently small depending 
 on the mesh size. Then the initial guess is sufficiently close to the unknown solution.
 
 \item The need for a good initial guess can be relaxed by using step size control, allowing a bigger time 
 step size. Anderson acceleration applied as post-processing can be interpreted as such, for more details, 
 cf. Section~\ref{section:acceleration}.
  
 \item Bounded derivatives of constitutive laws have to be available. In practice, nonlinearities employed in the 
 model~\eqref{model:p}--\eqref{model:u} are not necessarily Lipschitz continuous, e.g., the relative 
 permeability for soils. In particular, in the transition from the partially to the fully saturated
 regime, the derivative of the relative permeability modeled by the van Genuchten model~\cite{vanGenuchten1980} can be unbounded.
 Consequently, the Jacobian might become ill-conditioned.
 
 \item Due to the coupled nature of the problem, the linear system resulting from the problem~\eqref{newton:p}--\eqref{newton:u}
 has a saddle point structure and is therefore ill-conditioned. Hence, an advanced solver architecture is required.
 In the context of Biot's equations, the application of a fixed-stress type solver or preconditioner~\cite{Castelleto2015} yields remedy.
 
\end{itemize}




\subsection{Fixed-Stress-L-scheme -- A Picard-type simultaneous linearization and splitting}\label{section:fsl-scheme}

We propose a novel, robust linearization scheme for Eq.~\eqref{discretization:p}--\eqref{discretization:u}.
It is essentially a simultaneous application of the L-scheme linearization for the Richards equation, cf.~e.g.~\cite{List2016},
and the fixed-stress splitting scheme for the linear Biot equations, cf.~e.g.~\cite{Kim2009}, both utilizing diagonal stabilization.
In the following, we refer to the scheme as Fixed-Stress-L-scheme (FSL).
As derived in Section~\ref{section:fsl:theory}, it can be interpreted as L-scheme linearization of the nonlinear Biot
equations reduced to a pure pressure formulation or alternatively as nonlinear Gauss-Seidel-type solver, 
consisting of cheap iterations allowing separate, sophisticated simulators for the mechanical and flow subproblems.
In Section~\ref{section:fsl:theory}, we show convergence of the Fixed-Stress-L-scheme under physical assumptions.
Before defining the Fixed-Stress-L-scheme, we recall main ideas of both the L-scheme and the fixed-stress splitting scheme.

\paragraph{Main ideas of the L-scheme}
The L-scheme is an inexact Newton's method, employing constant linearization for monotone
and Lipschitz continuous terms. For remaining contributions Picard-type linearization is applied.
Effectively, this approach is identical with applying a Picard iteration with
diagonal stabilization. All in all, no explicit derivatives are required at the price
that only linear convergence can be expected. Under mild conditions, this concept has
been rigorously proven to be globally convergent for various porous media 
applications, e.g.,~\cite{List2016,Radu2014,Radu2017,Borregales2017}. Moreover, as pointed out
by~\cite{List2016}, the resulting linear problem is expected to be significantly better
conditioned than the corresponding linear problem obtained by Newton's method.

Regarding the nonlinear Biot equations, the saturation $\s=\s(p)$ is classically non-decreasing 
and Lipschitz continuous. Hence, assuming that $\phi^{i-1}\geq0$ on $\Omega$, the above criteria apply
to the saturation contribution in Eq.~\eqref{newton:p}. The approximation of the saturation at iteration 
$i$ is then given by
\begin{align} \label{linearization:l-scheme:saturation}
  \s^n \= \s(\ph^n) \,\approx\, \s(\pn{i}) \,\approx\, \s(\pn{i-1}) + L (\pn{i} - \pn{i-1}) \= \s(\pn{i-1}) + L \dpn{i},
\end{align}
where $L\in\mathbb{R}_+$ is a sufficiently large tuning parameter, usually set equal to the Lipschitz 
constant $\LipschitzS$ of $\s$. As coupling terms are not monotone, Picard-type linearization is applied
to the remaining contributions of Eq.~\eqref{discretization:p}--\eqref{discretization:u}.

\paragraph{Main ideas of the fixed-stress splitting scheme}
Considering the linear Biot equations, their linearization results in a saddle point problem, 
thus, requiring an advanced solver technology for efficient solution. 
For this purpose, physically motivated, robust, iterative splitting schemes are widely-used 
as, e.g., the fixed-stress splitting scheme, originally introduced by~\cite{Settari1998}.
As it decouples mechanics and flow equations, separate simulators can be utilized 
for both subproblems, reducing the complexity to solving simpler, better conditioned problems. 
The robust decoupling is accomplished via sufficient diagonal stabilization, 
introducing a tuning parameter $\beta_\mathrm{FS}$. Its optimization with guaranteed convergence
is a research question on its own~\cite{Mikelic2013,Both2017,Both2018}. Concepts can be also extended to
multiphase flow coupled with linear elasticity~\cite{Kim2013}.

\paragraph{Scheme}
We observe that applying the monolithic L-scheme, as just explained to the nonlinear, discrete
Biot equations~\eqref{discretization:p}--\eqref{discretization:u}, results in a linear problem
equivalent with that for single phase flow in heterogeneous media, for which the fixed-stress
splitting scheme is an attractive solver~\cite{Both2017}. Both schemes are realized via diagonal
stabilization. Anticipating the dynamics to be mainly governed by the flow problem, cf.\ Assumption~(A4), and the 
mechanics problem to be much simpler, a simultaneous application of the L-scheme and the 
fixed-stress splitting scheme yields an attractive linearization scheme incorporating the decoupling
of flow and mechanics equations. 

Written as iterative scheme in incremental form, the resulting Fixed-Stress-L-scheme reads:
For each time step, given the initial guess $\pquh{n,0}=\pquh{n-1}$, loop over the iterations $i\in\mathbb{N}$ until 
convergence is reached. For each iteration $i$, perform two steps:
\newline

\begin{indentblock}
\textit{1. Step:} Set $L=\LipschitzS$, the Lipschitz constant of $\s$, and 
$\beta_\mathrm{FS}=\alpha^2 / \left(\tfrac{2\mu}{d} + \lambda\right)$. 
Given $\pquh{n,i-1},\pquh{n-1}\in\QZVh$, find the increments $\dpqh{n,i}\in\QZh$, satisfying, for all $\testpqh\in\QZh$,
\end{indentblock}
\begin{align}
  \left\langle \left(L + \tfrac{1}{N} + \beta_\mathrm{FS}\right) \dpn{i}, \testph \right\rangle + \tau \langle \DIV \dqn{i}, \testph \rangle 	&\= r_p^{n,i-1}(\testph), 	\label{fsl:p} \\
  \langle \permeability(\sn{i-1})^{-1} \dqn{i}, \testqh \rangle  - \langle \dpn{i}, \DIV \testqh \rangle 				                &\= r_q^{n,i-1}(\testqh),       \label{fsl:q}
\end{align}
\begin{indentblock}
and set
\end{indentblock}
\begin{align*}
 \pqh{n,i} \= \pqh{n,i-1} + \dpqh{n,i}.
\end{align*}

\vspace{0.1cm}

\begin{indentblock}
\textit{2. Step:} 
Given $(\pqh{n,i},\un{i-1})\in\QZVh$,
find the increment $\dun{i}\in\Vh$, satisfying, for all $\testuh\in\Vh$,
\end{indentblock}
\begin{align}
  2\mu\langle \eps{\dun{i}}, \eps{\testuh} \rangle + \lambda \langle \DIV \dun{i}, \DIV \testuh \rangle &\= r_u^{n,i/i-1}(\testuh), \label{fsl:u} 
\end{align}
\begin{indentblock}
and set
\end{indentblock}
\begin{align*}
  \un{i} &\= \un{i-1} + \dun{i}.
\end{align*}
After convergence is reached at iteration $N$, set $\pquh{n}=\pquh{n,N}$.

\paragraph{Properties}
The Fixed-Stress-L-scheme inherits its properties from the underlying methods.
It does not require the evaluation of any derivatives, increasing the speed of the assembly process.
It is very robust but guarantees only linear convergence, as shown in 
Section~\ref{section:fsl:theory}, cf.\ Theorem~\ref{theorem:convergence:fsl}.
Furthermore, the Fixed-Stress-L-scheme requires solely the solution of the mechanical and flow problem, 
allowing separate simulators. In particular, the overall method utilizes a single loop in 
contrast to the Newton's method combined with a fixed-stress splitting scheme as iterative
solver.

\subsection{Quasi-Newton modifications of the Fixed-Stress-L-scheme}\label{section:fs-mp-newton}

The Fixed-Stress-L-scheme employs constant linearization for the fluid volume $\phi\s$ with respect to fluid pressure,
utilizing an upper bound for the Lipschitz constant. In many practical situations, this approach is quite
pessimistic. Recalling the assumption that the flow problem dominates the dynamics of the system,
we expect the simultaneous application of the fixed-stress splitting scheme and more sophisticated 
flow linearizations to be only slightly less robust than the Fixed-Stress-L-scheme. Independent of the 
flow linearization, diagonal stabilization is added by the splitting scheme increasing the robustness.
In the following, based on the derivation of the Fixed-Stress-L-scheme in Section~\ref{section:fsl:theory}, 
cf.\ Remark~\ref{remark:choice-l}, we couple simultaneously a modified Picard method~\cite{Celia1990} 
and Newton's method with the fixed-stress splitting scheme yielding the Fixed-Stress-Modified-Picard 
method and the Fixed-Stress-Newton method, respectively. The modified Picard method, in particular, is a 
widely-used linearization scheme for the Richards equation and hence rises interest for its use for the
linearization of the discrete, nonlinear Biot equations~\eqref{discretization:p}--\eqref{discretization:u}.

\paragraph{Fixed-Stress-Modified-Picard method}
Applied to the Richards equations, the modified Picard method employs a first order Taylor approximation as
linearization for the saturation and a Picard-type linearization for the possibly H\"older continuous 
permeability. By employing a first order approximation of the fluid volume $\phi\s$ with respect to 
fluid pressure instead, and by coupling simultaneously with the fixed-stress splitting scheme, we 
obtain a linearization scheme for Eq.~\eqref{discretization:p}--\eqref{discretization:u}. For later
reference, we denote the resulting scheme by Fixed-Stress-Modified-Picard-scheme. It is essentially identical
with the Fixed-Stress-L-scheme but with modified first fixed-stress step (1.~Step). We exchange 
Eq.~\eqref{fsl:p}--\eqref{fsl:q} with
\begin{align}
  \left\langle \left(\phi^{i-1} \frac{\partial \s}{\partial \pw}(\pn{i-1}) 
  + \left( \tfrac{1}{N} + \beta_\mathrm{FS}\right) (\sn{i-1})^2\right)  \dpn{i}, \testph \right\rangle 
  + \tau \langle \DIV \dqn{i}, \testph \rangle 								        &\= r_p^{n,i-1}(\testph), 	\label{fsmp:p}\\
  \langle \permeability(\sn{i-1})^{-1} \dqn{i}, \testqh \rangle  - \langle \dpn{i}, \DIV \testqh \rangle 	&\= r_q^{n,i-1}(\testqh).    \label{fsmp:q}
\end{align}

\paragraph{Fixed-Stress-Newton method}
In case the permeability is Lipschitz continuous, the simultaneous application of the fixed-stress
splitting scheme and linearization of the flow equations via Newton's method yields an attractive
linearization scheme for Eq.~\eqref{discretization:p}--\eqref{discretization:u}. For later reference, we denote 
the resulting scheme by Fixed-Stress-Newton method. It is essentially identical with the Fixed-Stress-L-scheme but 
with modified first fixed-stress step (1.~Step). We exchange Eq.~\eqref{fsl:p}--\eqref{fsl:q} with
\begin{align}
  \left\langle \left(\phi^{i-1} \frac{\partial \s}{\partial \pw}(\pn{i-1}) 
  + \left( \tfrac{1}{N} + \beta_\mathrm{FS}\right) (\sn{i-1})^2\right)  \dpn{i}, \testph \right\rangle 
  + \tau \langle \DIV \dqn{i}, \testph \rangle 						&\= r_p^{n,i-1}(\testph),    \label{fsn:p}\\
  \langle \permeability(\sn{i-1})^{-1} \dqn{i}, \testqh \rangle  
  + \left\langle \left(\left.\frac{\partial}{\partial \pw}\permeability(\sn{i-1})\right|_{\pn{i-1}}\right)^{-1} \qn{i-1} \dpn{i}, \testqh \right\rangle 
  - \langle \dpn{i}, \DIV \testqh \rangle 						&\= r_q^{n,i-1}(\testqh).    \label{fsn:q}
\end{align}
We note that the Fixed-Stress-Newton method is also closely related to applying a single fixed-stress iteration 
as inexact solver for the linear problem~\eqref{newton:p}--\eqref{newton:u} arising from Newton's method.

\subsection{$L^2(\Omega)$-type stopping criterion}
For the numerical examples in Section~\ref{section:numerical-results}, we employ a combination of an 
absolute and a relative $L^2(\Omega)$-type stopping criterion, closely related to the standard algebraic $l^2$-type 
criterion. Given tolerances $\varepsilon_\mathrm{a},\ \varepsilon_\mathrm{r}\in\mathbb{R}_+$, we denote an 
iteration as converged if it holds
\begin{align*}
 \| \incph^{n,i} \|_{L^2(\Omega)} + \| \incqh^{n,i} \|_{L^2(\Omega)} + \| \incuh^{n,i} \|_{L^2(\Omega)} < \varepsilon_\mathrm{a},
 \qquad \text{and} \qquad 
 \frac{\| \incph^{n,i} \|_{L^2(\Omega)}}{ \| p_h^{n,i} \|_{L^2(\Omega)}} 
 + \frac{\| \incqh^{n,i} \|_{L^2(\Omega)}}{ \| \vec{q}_h^{n,i} \|_{L^2(\Omega)}} 
 + \frac{\| \incuh^{n,i} \|_{L^2(\Omega)}}{ \| \vec{u}_h^{n,i} \|_{L^2(\Omega)}}
 < \varepsilon_\mathrm{r}. 
\end{align*}

\section{Convergence theory for simultaneous linearization and splitting via the L-scheme}\label{section:fsl:theory}

In the following, we show convergence of the Fixed-Stress-L-scheme~\eqref{fsl:p}--\eqref{fsl:u} under mild, physical assumptions.
For this purpose, we formulate the nonlinear discrete problem~\eqref{discretization:p}--\eqref{discretization:u}
as an algebraic problem, reduce the problem to a pure pressure problem by exact inversion and apply the L-scheme as linearization identical to the 
Fixed-Stress-L-scheme~\eqref{fsl:p}--\eqref{fsl:u}. Convergence follows then from an abstract convergence result.
For simplicity, we assume vanishing initial data and a homogeneous and isotropic material.

\paragraph{Algebraic formulation of the nonlinear, discrete Biot equations}

Given finite element bases for $\QZVh$, the nonlinear, discrete Biot 
equations~\eqref{discretization:p}--\eqref{discretization:u} translate to the algebraic equations
\begin{align}
 \MassSaturation(\pV{}) \left( \phiV + \alpha \DivPU \uV{} + \tfrac{1}{N} \MassMatrixP{}\peV{}(\pV{}) \right) + \tau \DivPQ \qV{} 	&= \rhsVPressure   \label{algebraic:nonlinear:three-field:p} \\
 \MassPermeability(\pV{})^{-1}\qV{} - \DivPQ^\top \pV{}                                							&= \rhsVFlux       \label{algebraic:nonlinear:three-field:q} \\
 \StiffnessMechanics \uV{} - \alpha\DivPU^\top \peV{}(\pV{})                         							&= \rhsVMechanics.  \label{algebraic:nonlinear:three-field:u}
\end{align}
%
We omit the detailed definition of the finite element matrices and vectors used in 
Eq.~\eqref{algebraic:nonlinear:three-field:p}--\eqref{algebraic:nonlinear:three-field:u},
as they are assembled in a standard way. We comment solely on their origin and their properties 
relevant for further discussion.

\begin{itemize}

 \item Let $\pV{}\in\mathbb{R}^{\NP},\ \qV{}\in\mathbb{R}^{\NQ},\ \uV{}\in\mathbb{R}^{\NU}$ denote the algebraic 
 pressure, volumetric flux and displacement coefficient vectors corresponding to $\pquh{n}\in\QZVh$ with respect
 to the chosen bases. 

 \item Let $\MassMatrixP{}\in\mathbb{R}^{\NP \times \NP}$, $\MassMatrixQ{}\in\mathbb{R}^{\NQ \times \NQ}$ be the natural 
 mass matrices incorporating local mesh information for $\mathcal{T}_h$ such that $\|\pV{}\|_{\MassMatrixP{}}= \|\ph \|_{L^2(\Omega)}$ 
 and $\|\qV{}\|_{\MassMatrixQ{}}= \|\qh \|_{L^2(\Omega)}$ for $\pqh{}\in\QZh$ and corresponding coefficient vectors 
 $\pV{}\in\mathbb{R}^{\NP}$, $\qV{}\in\mathbb{R}^{\NQ}$. In the following, let $\langle \cdot, \cdot \rangle$
 denote the classical $l^2$-vector scalar product $\mathbb{R}^{\star}\times\mathbb{R}^{\star}\rightarrow \mathbb{R}$
 with $\star\in\{\NP,\NQ,\NU\}$. Furthermore, for symmetric, positive definite matrices $\mathbf{M}\in\mathbb{R}^{\star\times\star}$,
 let $\|\cdot \|_\mathbf{M}$ be defined by $\|\mathbf{v} \|_\mathbf{M}^2= \langle \mathbf{M}\mathbf{v},\mathbf{v} \rangle$,
 $\mathbf{v}\in\mathbb{R}^\star$.
  
 \item Let $\MassSaturation:\mathbb{R}^{\NP} \rightarrow \mathbb{R}^{\NP \times \NP}$ denote a
 diagonal matrix with element-wise saturation $\s$ on the diagonal, i.e., $\MassSaturation(\pV{})_{kk}= \s(\pV{}_k)$ 
 for $\pV{}\in\mathbb{R}^{\NP}$, $k\in\{1,...,\NP\}$. 
 
 \item Let $\DivPU\in\mathbb{R}^{\NP \times \NU}$ and $\DivPQ\in\mathbb{R}^{\NP\times \NQ}$ denote the matrices
 corresponding to the divergence operating on displacement and volumetric flux spaces, respectively,
 mapping into the pressure space. Let local mesh information be incorporated, analog to the mass matrix~$\MassMatrixP{}$.
 
 \item Let $\peV:\mathbb{R}^{\NP}\rightarrow\mathbb{R}^{\NP}$ correspond to the element-wise equivalent pore
 pressure $\ppore$. For given $\pV{}\in\mathbb{R}^{\NP}$, each component of $\peV$ is given by
 $\peV(\pV{})_k= \ppore(\pV{}_k)$, $k\in\{1,...,\NP\}$.
 
 \item Let $\phiV\in\mathbb{R}^{\NP}$ denote the initial porosity vector incorporating local mesh information
 such that $\phiV + \alpha \DivPU \uV{} + \tfrac{1}{N} \MassMatrixP{}\peV{}(\pV{})$ 
 corresponds element-wise to the actual porosity of a deformed material.
 
 \item Let $\MassPermeability^{-1}:\mathbb{R}^{\NP}\rightarrow \mathbb{R}^{\NQ\times\NQ}$ denote the volumetric flux 
 mass matrix, weighted by the nonlinear permeability contribution $\permeability^{-1}(\s)$ in Darcy's law,
 and incorporating local mesh information. 
 
 \item Let $\StiffnessMechanics\in\mathbb{R}^{\NU \times \NU}$ denote the stiffness matrix, corresponding to 
 the linear elasticity equations, incorporating local mesh information.

 \item $\rhsVPressure\in\mathbb{R}^{\NP}$, $\rhsVFlux\in\mathbb{R}^{\NQ}$ and $\rhsVMechanics\in\mathbb{R}^{\NU}$
 incorporate solution independent contributions as volume effects and Neumann boundary conditions
 and data at the previous time step. Furthermore, let local mesh information be incorporated.
   
\end{itemize}

\paragraph{Compact formulation of the algebraic problem}

By inverting exactly Eq.~\eqref{algebraic:nonlinear:three-field:q} and Eq.~\eqref{algebraic:nonlinear:three-field:u} 
with respect to $\qV{}$ and $\uV{}$, respectively, and insert into Eq.~\eqref{algebraic:nonlinear:three-field:p},
we obtain the equivalent, reduced problem for $\pV{}$
\begin{align}\label{algebraic:nonlinear:one-field:p}
 \MassSaturation(\pV{}) \left( \phiV + \alpha \DivPU \StiffnessMechanics^{-1} \rhsVMechanics + \alpha^2 \DivPU \StiffnessMechanics^{-1} \DivPU^\top \peV{}(\pV{}) + \tfrac{1}{N} \MassMatrixP{}\peV(\pV{}) \right)
 + \tau \DivPQ \MassPermeability(\pV{}) \rhsVFlux  + \tau \DivPQ \MassPermeability(\pV{})  \DivPQ^\top \pV{} &= \rhsVPressure.
\end{align}

\noindent
By defining
\begin{align}\label{algebraic:nonlinear:phipV:definition}
 \phipV(\pV{}) &= \phiV + \alpha \DivPU \StiffnessMechanics^{-1} \rhsVMechanics + \alpha^2 \DivPU \StiffnessMechanics^{-1} \DivPU^\top \peV{}(\pV{}) + \tfrac{1}{N} \MassMatrixP{}\peV(\pV{}), \\
 \label{algebraic:nonlinear:bV:definition}
 \bV(\pV{})    &= \MassSaturation(\pV{}) \phipV(\pV{}),
 \qquad
 \DM = \DivPQ,
 \qquad
 \KM(\pV{}) = \MassPermeability(\pV{}),
\end{align}
the reduced problem~\eqref{algebraic:nonlinear:one-field:p} can be written in compact form
\begin{align}\label{algebraic:nonlinear:abstract}
 \bV(\pV{}) + \tau \DM \KM(\pV{}) \left(\rhsVFlux + \DM^{\top} \pV{}\right) &= \rhsVPressure.
\end{align}

\paragraph{L-scheme linearization}

We linearize the abstract problem~\eqref{algebraic:nonlinear:abstract} using the L-scheme,
introducing a sequence $\{\pV{i}\}_i\subset\mathbb{R}^{\NP}$ approximating the exact solution $\pV{}\in\mathbb{R}^{\NP}$.
Given a user-defined parameter $L\in\mathbb{R}_+$, we set $\LMatrix= L \MassMatrixP{}$. Then given
an initial guess $\pV{0}\in\mathbb{R}^{\NP}$, the scheme is defined as follows: 
Loop over the iterations $i\in\mathbb{N}$ until convergence is reached. At iteration $i$, 
given data $\pV{i-1}\in\mathbb{R}^{\NP}$, find $\pV{i}\in\mathbb{R}^{\NP}$ solving the linear problem
\begin{align}\label{algebraic:l-scheme:abstract}
 \LMatrix(\pV{i} - \pV{i-1}) + \bV(\pV{i-1}) + \tau \DM \KM(\pV{i-1}) \left(\rhsVFlux + \DM^{\top} \pV{i}\right) &= \rhsVPressure.
\end{align}

\begin{remark}[Equivalence to the Fixed-Stress-L-scheme]\label{remark:algebraic:l-scheme:fsl-scheme}
The L-scheme~\eqref{algebraic:l-scheme:abstract} is equivalent with the Fixed-Stress-L-scheme~\eqref{fsl:p}--\eqref{fsl:u}.
Indeed, exact inversion of Eq.~\eqref{fsl:q} and Eq.~\eqref{fsl:u} with respect to $\qh$ and $\uh$, respectively, insertion
into Eq.~\eqref{fsl:p} yields Eq.~\eqref{algebraic:l-scheme:abstract} after translation into an algebraic context.
The derivation in particular reveals the close connection of the fixed-stress splitting scheme and the L-scheme.
\end{remark}

\begin{lemma}[Convergence of the L-scheme]\label{lemma:abstract:l-scheme:convergence}
 Assume~\eqref{algebraic:nonlinear:abstract} and~\eqref{algebraic:l-scheme:abstract} both have unique solutions
 $\pV{}\in\mathbb{R}^{\NP}$ and $\pV{i}\in\mathbb{R}^{\NP}$, respectively. Furthermore, let the following assumptions 
 be satisfied:

 \begin{itemize} 

 \item[(L1)] There exists a constant $\LipschitzB\in\mathbb{R}_+$ satisfying 
 $\| \bV(\pV{}) - \bV(\tilde{\pV{}}) \|_{\MassMatrixP{-1}}^2 \leq \LipschitzB \langle \bV(\pV{}) - \bV(\tilde{\pV{}}) , \pV{} - \tilde{\pV{}}\rangle$
 for all $\pV{},\tilde{\pV{}}\in\mathbb{R}^{\NP}$, i.e., $\bV$ is in some sense monotonically increasing and Lipschitz continuous.
 
 \item[(L2)] There exist constants $\km,\kM\in\mathbb{R}_+$ satisfying $\km \| \qV{} \|_{\MassMatrixQ{-1}}^2 \leq \langle \KM(\pV{}) \qV{}, \qV{} \rangle \leq \kM \| \qV{} \|_{\MassMatrixQ{-1}}^2$
 for all $\pV{}\in\mathbb{R}^{\NP}$, $\qV{}\in\mathbb{R}^{\NQ}$. 
 Furthermore, there exists a constant $\LipschitzK$ satisfying 
 $\| (\KM(\pV{}) - \KM(\tilde{\pV{}}))\MassMatrixQ{}\|_{\MassMatrixQ{},\infty} \leq \LipschitzK \| \bV(\pV{}) - \bV(\tilde{\pV{}})\|_{\MassMatrixP{-1}}$
 for all $\pV{},\tilde{\pV{}}\in\mathbb{R}^{\NP}$, i.e. $\KM$ is in some sense Lipschitz continuous.
 Here, the subordinate matrix norm $\|\cdot \|_{{\MassMatrixQ{}},\infty}$ is defined
 by $\|\KM\|_{{\MassMatrixQ{}},\infty} = \underset{\qV{}\neq0}{\text{sup}}\,\| \KM \qV{} \|_{\MassMatrixQ{}} / \|\qV{}\|_\infty$, $\KM\in\mathbb{R}^{\NQ\times \NQ}$.
 
 \item[(L3)] There exists a constant $q_\infty\in\mathbb{R}_+$ satisfying $\| {\MassMatrixQ{-1}} \rhsVFlux + \DM^\top \pV{} \|_\infty \leq q_\infty$
 for the solution of problem~\eqref{algebraic:nonlinear:abstract}, i.e., boundedness is satisfied.
  
 \end{itemize}
 
 \noindent
 If the parameter $L$ and the time step size $\tau$ are chosen such that
 $\tfrac{2}{\LipschitzB} - \tfrac{1}{L} - \tau \tfrac{q_\infty^2 \LipschitzK^2}{2\km} \geq 0$, for a fixed constant $C_\Omega>0$, it holds
 \begin{align*}
  \| \pV{i} - \pV{} \|_{\MassMatrixP{}}^2 \leq \frac{L}{L + \tau \km C_\Omega^2} \| \pV{i-1} - \pV{} \|_{\MassMatrixP{}}^2.
 \end{align*}
\end{lemma}

\noindent
The proof of Lemma~\ref{lemma:abstract:l-scheme:convergence} is given in~\ref{section:appendix}. The proof
is essentially the same as for the Richards equation by~\cite{List2016} but formulated in a slightly more general framework.
Assumption~(L1)--(L2) are generalized versions of assumptions made in~\cite{List2016} due to the possible 
global dependence of each component of $\bV=\bV(\pV{})$ on $\pV{}$.

\paragraph{Consequence for the Fixed-Stress-L-scheme}
By Remark~\ref{remark:algebraic:l-scheme:fsl-scheme}, the Fixed-Stress-L-scheme~\eqref{fsl:p}--\eqref{fsl:u}
is equivalent with the L-scheme~\eqref{algebraic:l-scheme:abstract}. Therefore, we check 
Assumption~(L1)--(L3) of Lemma~\ref{lemma:abstract:l-scheme:convergence} particularly for 
Eq.~\eqref{algebraic:nonlinear:one-field:p} in order to analyze the Fixed-Stress-L-scheme. We make 
the following physical assumptions:

\begin{itemize}
 \item[(F1)] With the varying porosity $\phipV=\phipV(\pV{})$ as defined in Eq.~\eqref{algebraic:nonlinear:phipV:definition},
 let $\pVSpace = \left\{ \pV{}\in\mathbb{R}^{\NP} \left| \phipV(\pV{}) \in[0,1] \text{ component-wise} \right. \right\}$
 denote the space of all pressures leading to physical deformations.
 
 \item[(F2)] Let the saturation model $\s:\mathbb{R}\rightarrow[0,1]$ have a bounded derivative and assume a non-vanishing residual
 saturation $0<\smin=\underset{\pV{}\in\pVSpace,\ i\in\{1,...,\NP\}}{\text{inf}} s(\pV{}_i)$. 
 
 \item[(F3)] Let the material law $\permeability=\permeability(\s):[0,1]\rightarrow\mathbb{R}$
 be Lipschitz continuous and assume there exist constants $k_\mathrm{w,m},k_\mathrm{w,M}\in\mathbb{R}_+$ satisfying
 $k_\mathrm{w,m} \leq \permeability (\s(\pV{}_i)) \leq k_\mathrm{w,M}$ for all $\pV{}\in\pVSpace,\ i\in\{1,...,\NP\}$.
 
 \item[(F4)] There exists a constant $q_\infty\in\mathbb{R}_+$ satisfying $\| {\MassMatrixQ{-1}} \rhsVFlux + \DM^\top \pV{} \|_\infty \leq q_\infty$
 for the solution of problem~\eqref{algebraic:nonlinear:abstract}, i.e., essentially fluxes are bounded.
\end{itemize}

\noindent
 Assumption~(F2)--(F4) are standard assumptions generally accepted for the analysis of the Richards equation.
 In particular, if the assumptions are not satisfied, the Richards equation as model for flow
 in partially saturated porous media has to be questioned.
In order to show Assumption~(L1)--(L2), we first show that $\bV$ is bi-Lipschitz.

\begin{lemma}\label{lemma:algebraic:b:lipschitz}
 Let Assumption~(F1)--(F2) be satisfied. Then for $\bV$ as defined in Eq.~\eqref{algebraic:nonlinear:bV:definition}, 
 there exist mesh-independent constants $\lipschitzB,\LipschitzB\in\mathbb{R}_+$ satisfying for all $\pV{},\tilde{\pV{}}\in\pVSpace$
 \begin{align*}
  \lipschitzB \| \pV{} - \tilde{\pV{}} \|_{\MassMatrixP{}}^2 \leq \langle \bV(\pV{})  - \bV(\tilde{\pV{}}), \pV{} - \tilde{\pV{}} \rangle \leq \LipschitzB \| \pV{} - \tilde{\pV{}} \|_{\MassMatrixP{}}^2.
 \end{align*}
\end{lemma}

\begin{proof}
 As $\bV\in C^1(\mathbb{R}^{\NP};\mathbb{R}^{\NP})$, with Jacobian $\DbV(\pV{})\in\mathbb{R}^{\NP\times \NP}$, $\pV{}\in\mathbb{R}^{\NP}$,
 and $\MassMatrixP{}$ is a diagonal matrix, it holds
 \begin{align}\label{proof:algebraic:b:lipschitz:1}
  \underset{\substack{ \pV{},\tilde{\pV{}}\in\pVSpace \\ \pV{}\neq\tilde{\pV{}}}}{\text{sup}}
  \frac{\langle \bV(\pV{})  - \bV(\tilde{\pV{}}), \pV{} - \tilde{\pV{}} \rangle }{\| \pV{} - \tilde{\pV{}} \|_{\MassMatrixP{}}^2} 
  = 
  \underset{\substack{ \pV{}\in\pVSpace, \hV\in\mathbb{R}^{\NP}\setminus\{\nullV\} \\ \pV{}+\hV\in\pVSpace}}{\text{sup}}
  \frac{\langle \DbV(\pV{}) \hV , \hV \rangle }{\| \hV \|_{\MassMatrixP{}}^2}
  =
  \underset{\substack{ \pV{}\in\pVSpace, \hV\in\mathbb{R}^{\NP}\setminus\{\nullV\} \\ \pV{}+\hV\in\pVSpace}}{\text{sup}}
  \frac{\langle \MassMatrixP{-1/2}\DbV(\pV{})\MassMatrixP{-1/2} \hV , \hV \rangle }{\| \hV \|^2}.
 \end{align}
 Employing the properties of $\bV$, and making use of the specific choice of the equivalent pore pressure~\eqref{model:pore-pressure:definition},
 the Jacobian of $\bV$ is given by
 \begin{align}\label{proof:algebraic:b:lipschitz:2}
  \DbV(\pV{}) = \begin{bmatrix}
                 s'(\pV{}_1) \phipV_1(\pV{}) &  & \\
                 & \ddots & \\
                 & & s'(\pV{}_{\NP}) \phipV_{\NP}(\pV{})
                \end{bmatrix}
                +
                \alpha^2\MassSaturation(\pV{}) \DivPU \StiffnessMechanics^{-1} \DivPU^\top \MassSaturation(\pV{})^\top
                +
                \frac{1}{N} \MassSaturation(\pV{})\MassMatrixP{}\MassSaturation(\pV{})^\top.
 \end{align}
 Hence, $\DbV(\pV{})=\DbV(\pV{})^\top$ for all $\pV{}\in\pVSpace$ with eigenvalues greater or equal than zero. 
 Consequently, the largest value for the Rayleigh quotient~\eqref{proof:algebraic:b:lipschitz:1} is given by 
 the largest eigenvalue of $\MassMatrixP{-1/2}\DbV(\pV{})\MassMatrixP{-1/2}$ maximized over $\pV{}\in\pVSpace$.
 
 The porosity vector $\phiV$ is essentially scaled by $\MassMatrixP{}$.
 Additionally, as shown by~\cite{Adler2017}, $\DivPU \StiffnessMechanics^{-1} \DivPU^\top$ is norm equivalent with $\MassMatrixP{}$. 
 Hence, also $\DbV(\pV{})$ is norm equivalent with the standard mass matrix with mesh-independent 
 bounds. Together with employing the assumptions, we see there exists a largest eigenvalue 
 $L_b\in\mathbb{R}_+$ of $\MassMatrixP{-1/2}\DbV\MassMatrixP{-1/2}$ independent of the mesh.
 Analogously, it holds
 \begin{align*}
  \underset{\substack{ \pV{},\tilde{\pV{}}\in\pVSpace \\ \pV{}\neq\tilde{\pV{}}}}{\text{inf}}
  \frac{\langle \bV(\pV{})  - \bV(\tilde{\pV{}}), \pV{} - \tilde{\pV{}} \rangle }{\| \pV{} - \tilde{\pV{}} \|_{\MassMatrixP{}}^2}
  = 
  \underset{\substack{ \pV{}\in\pVSpace, \hV\in\mathbb{R}^{\NP}\setminus\{\nullV\} \\ \pV{}+\hV\in\pVSpace}}{\text{inf}}
  \frac{\langle \MassMatrixP{-1/2}\DbV(\pV{})\MassMatrixP{-1/2} \hV , \hV \rangle }{\| \hV \|^2}
 \end{align*}
 with the value given by the smallest eigenvalue $l_b$ of $\MassMatrixP{-1/2}\DbV(\pV{})\MassMatrixP{-1/2}$ minimized 
 over $\pV{}\in\pVSpace$. From above discussion it follows that $l_b\in\mathbb{R}_+$ is mesh-independent.
 All in all, the proposed thesis follows.
\end{proof}

\begin{remark}[Parabolic character of the nonlinear Biot equations]\label{remark:model:properties}
The Richards equation itself is a degenerate elliptic-parabolic equation due to possible development 
of fully saturated regions. However, from Eq.~\eqref{proof:algebraic:b:lipschitz:2} it follows, that
this type of degeneracy is not adopted by the nonlinear Biot equations~\eqref{model:p}--\eqref{model:u}.
Independent of the mesh size, the derivative of the fluid volume $\phi\s$ with respect to fluid pressure 
is not vanishing, as long as the fluid saturation is not vanishing. This observation is consistent with 
considerations by~\cite{Showalter2001} on the classical, linear Biot equations.
We note for weak coupling of mechanics and flow equations, numerically the parabolic character might
be effectively lost, making the original two-way coupled problem essentially equivalent to the Richards equation,
one-way coupled with the linear elasticity equations.
\end{remark}

\begin{corollary}
 Let Assumption~(F1)--(F2) be satisfied. Then $\bV$ is invertible on $\pVSpace$.
\end{corollary}

\begin{corollary}
 Let Assumption~(F1)--(F2) be satisfied. Then Assumption~(L1) is satisfied, in the sense, 
 for all $\pV{},\tilde{\pV{}}\in\pVSpace$, it holds
 \begin{align*}
  \| \bV(\pV{}) - \bV(\tilde{\pV{}}) \|_{\MassMatrixP{-1}}^2 \leq \LipschitzB \langle \bV(\pV{}) - \bV(\tilde{\pV{}}), \pV{} - \tilde{\pV{}} \rangle.
 \end{align*}
\end{corollary}

\begin{proof}
As $\bV$ is invertible and $\DbV$ is symmetric, using the Inverse Function theorem, it holds
\begin{align*}
  \underset{\substack{ \pV{},\tilde{\pV{}}\in\pVSpace \\ \pV{}\neq\tilde{\pV{}}}}{\text{sup}}
  \frac{\| \bV(\pV{}) - \bV(\tilde{\pV{}}) \|_{\MassMatrixP{-1}}^2 }{\langle \bV(\pV{}) - \bV(\tilde{\pV{}}), \pV{} - \tilde{\pV{}} \rangle}  
  &=
  \underset{\substack{ \bV^{-1}(\pV{}),\bV^{-1}(\tilde{\pV{}})\in\pVSpace \\ \pV{}\neq\tilde{\pV{}}}}{\text{sup}}
  \frac{\| \pV{} - \tilde{\pV{}} \|_{\MassMatrixP{-1}}^2 }{\langle \bV^{-1}(\pV{}) - \bV^{-1}(\tilde{\pV{}}), \pV{} - \tilde{\pV{}} \rangle} \\
  &=
  \left[ \underset{\substack{ \bV^{-1}(\pV{}),\bV^{-1}(\tilde{\pV{}})\in\pVSpace \\ \pV{}\neq\tilde{\pV{}}}}{\text{inf}}
  \frac{\langle \bV^{-1}(\pV{}) - \bV^{-1}(\tilde{\pV{}}), \pV{} - \tilde{\pV{}} \rangle}{\| \pV{} - \tilde{\pV{}} \|_{\MassMatrixP{-1}}^2}\right]^{-1} \\
  &=
  \underset{\substack{ \pV{},\tilde{\pV{}}\in\pVSpace \\ \pV{}\neq\tilde{\pV{}}}}{\text{sup}}
  \frac{\langle \bV(\pV{})  - \bV(\tilde{\pV{}}), \pV{} - \tilde{\pV{}} \rangle }{\| \pV{} - \tilde{\pV{}} \|_{\MassMatrixP{}}^2}.
\end{align*}
The result follows from Lemma~\ref{lemma:algebraic:b:lipschitz}. 
\end{proof}

\noindent
Analogously, we obtain:
\begin{corollary}\label{corollary:algebraic:b:auxiliary}
 Let Assumption~(F1)--(F2) be satisfied. Then for all $\pV{},\tilde{\pV{}}\in\pVSpace$, it holds
 \begin{align*}
  \lipschitzB \langle \bV(\pV{}) - \bV(\tilde{\pV{}}), \pV{} - \tilde{\pV{}} \rangle \leq \| \bV(\pV{}) - \bV(\tilde{\pV{}}) \|_{\MassMatrixP{-1}}^2.
 \end{align*}
\end{corollary}

\begin{corollary}\label{corollary:algebraic:k:lipschitz}
 Let Assumption~(F1)--(F3) be satisfied. Then, Assumption~(L2) is satisfied, in the sense, 
 there exists a constant $\LipschitzK\in\mathbb{R}_+$, satisfying for all $\pV{},\tilde{\pV{}}\in\pVSpace$,
 \begin{align}\label{corollary:algebraic:k:lipschitz:1}
  \| (\KM(\pV{}) - \KM(\tilde{\pV{}}))\MassMatrixQ{}\|_{\MassMatrixQ{},\infty} \leq \LipschitzK \| \bV(\pV{}) - \bV(\tilde{\pV{}})\|_{\MassMatrixQ{-1}}. 
 \end{align}
 Furthermore, there exist constants $\km,\kM\in\mathbb{R}_+$, satisfying for all $\pV{}\in\mathbb{R}^{\NP}$, 
 $\qV{}\in\mathbb{R}^{\NQ}$
 \begin{align}\label{corollary:algebraic:k:lipschitz:2}
  \km \| \qV{} \|_{\MassMatrixQ{-1}}^2 \leq \langle \KM(\pV{}) \qV{}, \qV{} \rangle \leq \kM \| \qV{} \|_{\MassMatrixQ{-1}}^2.
 \end{align} 
 \end{corollary}

\begin{proof}
 As the underlying $\permeability=\permeability(\s)$ is Lipschitz continuous, together with a scaling argument, 
 it follows, there exists a constant $\LipschitzLambdaT\in\mathbb{R}_+$ satisfying
 \begin{align*}
  \| (\KM(\pV{}) - \KM(\tilde{\pV{}})) \MassMatrixQ{}\|_{\MassMatrixQ{},\infty} \leq \LipschitzLambdaT \| \MassSaturation(\pV{}) - \MassSaturation(\pV{}) \|_{\MassMatrixP{},\infty}.
 \end{align*}
 Furthermore, as $\s=\s(p)$ is Lipschitz continuous, and $\MassSaturation$ is a diagonal matrix, 
 there exists a constant $\LipschitzS\in\mathbb{R}_+$ satisfying
 \begin{align*}
  \| \MassSaturation(\pV{}) - \MassSaturation(\pV{}) \|_{\MassMatrixP{},\infty} \leq \LipschitzS \| \pV{} - \tilde{\pV{}} \|_{\MassMatrixP{}}.
 \end{align*}
 All in all, with Lemma~\ref{lemma:algebraic:b:lipschitz} and Corollary~\ref{corollary:algebraic:b:auxiliary}, 
 Eq.~\eqref{corollary:algebraic:k:lipschitz:1} follows with $\LipschitzK = \LipschitzLambdaT \LipschitzS \lipschitzB^{-2}$.
 Eq.~\eqref{corollary:algebraic:k:lipschitz:2} follows directly from Assumption~(F3) together with a scaling argument.
\end{proof}

\noindent
All in all, under the assumptions of non-vanishing residual saturation, permeability, and porosity,
we obtain convergence for the L-scheme~\eqref{algebraic:l-scheme:abstract}, which translates directly
to the Fixed-Stress-L-scheme~\eqref{fsl:p}--\eqref{fsl:u}.

\begin{theorem}\label{theorem:convergence:fsl}
 Let Assumption~(F1)--(F4) be satisfied. 
 Let $\pV{}\in\mathbb{R}^{\NP}$ and $\pV{i}\in\mathbb{R}^{\NP}$ be the solutions of the nonlinear problem~\eqref{algebraic:nonlinear:one-field:p}
 and the L-scheme~\eqref{algebraic:l-scheme:abstract}, respectively. Assume they are unique.
 Let the initial guess $\pV{0}\in\mathbb{R}^{\NP}$ satisfy $\mathcal{B}_{\pV{}}(\| \pV{0} - \pV{}\|_{\MassMatrixP{}}) \subset \pVSpace$,
 where $\mathcal{B}_{\pV{}}(r)\subset\mathbb{R}^{\NP}$ denotes the sphere with center $\pV{}$ and radius $r>0$.
 Let $L$ and $\tau$ be chosen such that $\tfrac{1}{\LipschitzB} - \tfrac{1}{2L} - \tau \tfrac{q_\infty^2 \LipschitzK^2}{2\km}\geq0$.
 Then the L-scheme~\eqref{algebraic:l-scheme:abstract} converges linearly with mesh-independent convergence rate $\sqrt{\tfrac{L}{L + \tau \km C_\Omega^2}}$.
 Furthermore, by induction, each iterate is a physical solution $\{ \pV{i} \}_i \subset \pVSpace$.
\end{theorem}

\begin{remark}[Choice of $L$]\label{remark:choice-l}
 Including knowledge on the convergence for the fixed-stress splitting scheme, 
 Eq.~\eqref{proof:algebraic:b:lipschitz:2} justifies the choice of the tuning parameter for the 
 Fixed-Stress-L-scheme, cf.\ Section~\ref{section:fsl-scheme}. Assuming the worst case scenario,
 all quantities are globally maximized yielding an \textit{a priori choice}. This pessimistic choice slows 
 down potential convergence but increases robustness. From the proof of Theorem~\ref{theorem:convergence:fsl},
 it follows that local optimization would be sufficient, yielding an optimal but solution-dependent tuning parameter.
 In this spirit, Eq.~\eqref{proof:algebraic:b:lipschitz:2} also provides the basis for the modification of the tuning
 parameter used for both the Fixed-Stress-Modified-Picard method and the Fixed-Stress-Newton method,
 cf.\ Section~\ref{section:fs-mp-newton}. 
\end{remark}

\begin{remark}[Limitations of the Fixed-Stress-L-scheme]\label{remark:limits-l}
 Based on Theorem~\ref{theorem:convergence:fsl}, we expect the convergence of the Fixed-Stress-L-scheme~\eqref{fsl:p}--\eqref{fsl:u}
 to deteriorate for either too large time steps or too large Lipschitz constants for the constitutive laws $\s$ and $\permeability$. 
 This applies in particular if the constitutive laws are only H\"older continuous. Furthermore, given the parameter $L$ 
 is sufficiently large and the time step size sufficiently small, theoretical convergence of the Fixed-Stress-L-scheme is
 guaranteed. However, in practice, numerical round-off errors might lead to stagnation.
\end{remark}

\section{Acceleration and stabilization by Anderson acceleration}\label{section:acceleration}

The Fixed-Stress-L-scheme is expected to be a linearly convergent fixed-point iteration 
with the convergence rate depending on a tuning parameter. Its considered modifications 
employ a less conservative choice for the tuning parameter with the risk of failing convergence.
Consequently, we are concerned with two issues -- slow convergence and robustness with 
respect to the tuning parameter.

All presented linearization schemes in Section~\ref{section:linearization} can be 
interpreted as fixed-point iterations $\xn{i}=\FP{\xn{i-1}}=\xn{i-1}+\dFP{\xn{i-1}}$,
where $\xn{i}$ denotes the algebraic vector associated with $\pquh{n,i}$ and 
$\dFP{\xn{i-1}}$ is the actual, computed increment within the linearization scheme.
For such in general, Anderson acceleration~\cite{Anderson1965} has been demonstrated 
on several occasions to be a suitable method to accelerate convergence. Furthermore,
due to its relation to preconditioned, nonlinear GMRES\cite{Walker2011}, we also expect Anderson 
acceleration to increase robustness with respect to the tuning parameter for the 
considered linearization schemes. Both properties are justified by theoretical 
considerations in Section~\ref{section:linear-analysis:restarted-aa} and demonstrated 
numerically in Section~\ref{section:numerical-results}.

\paragraph{Scheme}

The main idea of the Anderson acceleration applied to a fixed-point iteration is to 
utilize previous iterates and mix their contributions in order to obtain a new iterate.
The method is applied as post-processing not interacting with the underlying fixed-point 
iteration. In the following, we denote AA($m$) the Anderson acceleration reusing 
$m+1$ previous iterations, such that AA(0) is identical to the original fixed point iteration.
We can apply AA($m$) to post-process the presented linearization schemes.
In compact notation, the scheme reads:

\begin{algo}[\textbf{AA($m$) accelerated $\mathcal{FP}$}]\label{algorithm:standard-aa}{\ }
\begin{algorithmic}
\State {Given: $\mathcal{FP}$, $\xn{0}$}
\For {$i$=1,2,..., until convergence}
  \State {Define depth $m_i=\min\{i-1,m\}$}
  \State {Define matrix of increments $\mathbf{F}_i= \begin{bmatrix} \dFP{\xn{i-m_i-1}}, ... , \dFP{\xn{i-1}} \end{bmatrix}$}
  \State {Minimize $\| \mathbf{F}_i \bm{\alpha} \|_2$ wrt.\ $\bm{\alpha}\in\mathbb{R}^{m_i+1}$ s.t.\ $\sum_k\! \alpha_k=1$}
  \State {Define next iterate $\xn{i} \= \sum_{k=0}^{m_i} \alpha_k \FP{\xn{k+i-m_i-1}}$}
\EndFor
\end{algorithmic}
\end{algo}

\noindent
For the specific implementation, we follow Walker and Peng~\cite{Walker2011}.
In particular, in Step~4, we solve an equivalent unconstrained minimization
problem, which is better conditioned, relatively small and cheap.
The main price to be paid is the additional storage of the vectors
 $\begin{bmatrix} \dFP{\xn{i-m-1}}, ... , \dFP{\xn{i-1}} \end{bmatrix}$ and 
 $\begin{bmatrix} \FP{\xn{i-m-1}}, ... , \FP{\xn{i-1}} \end{bmatrix}$.

\paragraph{Properties}

As post-processing Anderson acceleration does not modify the character of the 
underlying method, i.e., a coupled or decoupled character remains unchanged. In particular,
in contrast to classical preconditioning, no monolithic simulator is required.
Hence, Anderson acceleration is an attractive method in order to accelerate splitting schemes.

In many practical applications, effective acceleration can be observed. Though, there is no general,
theoretical guarantee for the Anderson acceleration to accelerate convergence of an underlying, 
convergent fixed-point iteration. Theoretically, even divergence is possible~\cite{Walker2011}. 
In the literature, so far, theoretical convergence results are solely known for contractive 
fixed-point iterations~\cite{Toth2015}. For nonlinear problems, AA($m$) is locally r-linearly
convergent with theoretical convergence rate not larger than the original contraction constant if the 
coefficients $\bm{\alpha}$ remain bounded. Without assumptions on $\bm{\alpha}$, AA(1) converges globally,
q-linearly in case the contraction constant is sufficiently small. Both results only guarantee the lack 
of deterioration but not acceleration.

For a special, linear case, in Section~\ref{section:linear-analysis:restarted-aa}, we show global convergence 
and theoretical acceleration for a variant of AA(1), fortifying the potential of Anderson acceleration.
In particular, Corollary~\ref{corollary:aa:convergence} predicts the ability of the Anderson acceleration
to increase robustness, allowing non-contractive fixed-point iterations to converge.
This motivates to apply AA($m$) also to accelerate possibly diverging Newton-like methods as the monolithic 
Newton method and the Fixed-Stress-Newton method with the risk of loosing potential, quadratic convergence.

\section{Theoretical contraction and acceleration for the restarted Anderson acceleration}\label{section:linear-analysis:restarted-aa}

For a special linear case, we prove global convergence of a restarted version of the Anderson
acceleration. In particular, convergence for 
non-contractive fixed-point iterations and effective acceleration for a class of contractive 
fixed-point iterations is shown.

\subsection{Restarted Anderson acceleration}
The original Anderson acceleration AA($m$) constantly utilizes the full set of $m$ previous
iterates. By defining the depth $m_i^\star=\text\{i-1 \text{ mod }m+1,m\}$ in the first step of 
Algorithm~\ref{algorithm:standard-aa} and apart from that following the remaining steps, 
we define a restarted version AA$^\star$($m$) of AA($m$), closer related to GMRES($m$). 
In words, in each iteration we update the set of considered iterates by the most current 
iterate. And in case, the number of iterates becomes $m+1$, we flush the memory and restart 
filling it again. In particular, for $m=1$, the algorithm reads:

\begin{algo}[\textbf{AA$^\star$($1$) accelerated $\mathcal{FP}$}]\label{algorithm:restarted-aa}{\ }
\begin{algorithmic}
\State {Given: $\x{0}$}
\For {$i$=0,2,4,..., until convergence}
  \State {Set $\x{i+1} = \FP{\x{i}}$}
  \State {Minimize $\left\| \dFP{\x{i+1}} + \alpha^{i+1} (\dFP{\x{i}} - \dFP{\x{i+1}}) \right\|$ wrt.\ $\alpha^{(i+1)}\in\mathbb{R}$}
  \State {Set $\x{i+2} = \FP{\x{i+1}} + \alpha^{(i+1)} (\FP{\x{i}} - \FP{\x{i+1}})$}
\EndFor
\end{algorithmic}
\end{algo}

\noindent
From~\cite{Toth2015}, it follows directly, that for $\mathcal{FP}$, a linear contraction, AA$^\star$($1$)
converges globally with convergence rate at most equal the contraction constant of $\mathcal{FP}$. In the following,
we extend the result to a special class of non-contractions.

\subsection{Convergence result}
For the convergence results, cf.~Lemma~\ref{lemma:anderson-star-1} and 
Corollary~\ref{corollary:aa:acceleration},~\ref{corollary:aa:convergence},
we make the following assumptions:

\begin{itemize}
 
 \item[(C1)] $\FP{\x{}} = \A \x{} + \b$ defines the Richardson iteration for $(\Id - \A)\x{} = \b$, $\A\in\mathbb{R}^{n\times n}$, $n>1$, $\b\in\mathbb{R}^n$.
 
 \item[(C2)] $\A$ is symmetric, and hence, $\A$ is orthogonally diagonalizable and there 
             exists an orthogonal basis of eigenvectors $\{\mathbf{v}_j\}_j$ and a corresponding
             set of eigenvalues $\{\lambda_j\}_j$ satisfying $\A\mathbf{v}_j=\lambda_j \mathbf{v}_j$.
             
 \item[(C3)] There exists a unique $\x{\star}$ such that $\FP{\x{\star}} = \x{\star}$, i.e., $\Id-\A$ is invertible.
 
 \item[(C4)] The initial iterate $\x{0}$ is chosen such that the initial error $\x{0} - \x{\star}\in\text{span}\{ \mathbf{v}_1, \mathbf{v}_2\}$,
 where $\mathbf{v}_1,\mathbf{v}_2$ are two orthogonal eigenvectors of $\A$.
 To avoid a trivial case, we assume $\lambda_1,\lambda_2\neq 0$.
 
\end{itemize}

\noindent
Then we are able to relate the errors between iterations of AA$^\star$(1), allowing to prove further convergence and acceleration
results, cf. Corollary~\ref{corollary:aa:convergence} and Corollary~\ref{corollary:aa:acceleration}.
All in all, the proof employs solely elementary calculations. However, as we are not aware of a general result 
of same type in the literature, we present the proof.

\begin{lemma}[Main result]\label{lemma:anderson-star-1}
 Let the Assumption~(C1)--(C4) be satisfied and let $\{\x{i}\}_i$ define the sequence defined by AA$^\star$($1$)
 applied to $\mathcal{FP}$. Furthermore, let $\e{i}=\x{i}-\x{\star}$ denote the error. 
 Then it holds 
 \begin{align*}
  \| \e{i+4} \| \leq r(\lambda_1,\lambda_2) \| \e{i} \|, \quad i=0,4,8,12,...
 \end{align*}
 for
 \begin{align*}
  r(\lambda_1,\lambda_2) = \frac{\lambda_1^2\lambda_2^2 (\lambda_2 - \lambda_1)^2}{\left(|\lambda_1(\lambda_1-1)| + |\lambda_2(\lambda_2-1)|\right)^2}.
 \end{align*}
\end{lemma}

\begin{proof}
First, an iteration-dependent error propagation matrix is derived, and second, an upper
bound for its spectral radius is computed. For this purpose, let us ignore for a moment Assumption~(C4). 

\paragraph{Iterative error propagation}
As we intend to relate $\e{i+4}$ with $\e{i}$, we explicitly write out the first four iterates and the corresponding
errors. Given $\x{i}$, by using $\b = \x{\star} - \A \x{\star}$ and $\x{i} - \x{i+1} = \e{i} - \e{i+1}$, we obtain
\begin{align}
 \x{i+1} &= \A\x{i}   + \b,    					   & \e{i+1} &= \A\e{i},                                            \label{proof:definition:error-1} \\
 \x{i+2} &= \A\x{i+1} + \b + \alpha^{(i+1)} \A(\x{i} - \x{i+1}),   & \e{i+2} &= \A\e{i+1} + \alpha^{(i+1)} \A(\e{i} - \e{i+1}),     \label{proof:definition:error-2} \\
 \x{i+3} &= \A\x{i+2} + \b,                                        & \e{i+3} &= \A\e{i+2},                                          \label{proof:definition:error-3} \\
 \x{i+4} &= \A\x{i+3} + \b + \alpha^{(i+3)} \A(\x{i+3} - \x{i+2}), & \e{i+4} &= \A\e{i+3} + \alpha^{(i+3)} \A(\e{i+2} - \e{i+3}).   \label{proof:definition:error-4}
\end{align}
By plugging all together, we obtain
\begin{align*}
 \e{i+4} 
 &= \A(\A + \alpha^{(i+3)} (\Id-\A)) \A(\A + \alpha^{(i+1)} (\Id-\A)) \e{i}.
\end{align*}
It suffices to bound the largest eigenvalue of the error propagation matrix 
$\A(\A + \alpha^{(i+3)} (\Id-\A)) \A(\A + \alpha^{(i+1)} (\Id-\A))$. From Assumption~(C2) it follows that
$\{\mathbf{v}_i\}_i$ defines an orthogonal basis of eigenvectors for the error propagation matrix
with corresponding eigenvalues $\{\tilde{\lambda}_j\}_j$ defined by
\begin{align}\label{proof:eigenvalue:error-propagation}
 \tilde{\lambda}_j=\lambda_j^2 (\lambda_j + \alpha^{(i+1)} (1 - \lambda_j)) (\lambda_j + \alpha^{(i+3)} (1-\lambda_j)).
\end{align}

\paragraph{Explicit definition of $\alpha^{(i+1)}$ and $\alpha^{(i+3)}$}
The minimization problem in Algorithm~\ref{algorithm:restarted-aa} can be solved explicitly, by solving adequate normal equations.
It follows, that
\begin{align*}
 \alpha^{(i+1)} &= \frac{(\dFP{\x{i+1}} - \dFP{\x{i}}) \cdot \dFP{\x{i+1}}}{(\dFP{\x{i+1}} - \dFP{\x{i}}) \cdot(\dFP{\x{i+1}} - \dFP{\x{i}})}.
\end{align*}
After employing simple arithmetics and using Eq.~\eqref{proof:definition:error-1}, we obtain
\begin{align*}
 \dFP{\x{i+1}}               &= (\A-\Id) \x{i+1} + \b = (\A - \Id) \e{i+1}  = (\A - \Id) \A \e{i} = \A(\A - \Id) \e{i},\\
 \dFP{\x{i+1}} - \dFP{\x{i}} &= (\A-\Id) (\x{i+1} - \x{i}) = (\A - \Id) (\e{i+1} - \e{i}) = (\A - \Id)^2\e{i}.
\end{align*}
Consequently, it holds
\begin{align}\label{proof:alpha(k+1)}
 \alpha^{(i+1)} 
 = \frac{ ((\A - \Id)^2\e{i}) \cdot (\A(\A - \Id)\e{i})}{ \|(\A - \Id)^2\e{i}\|^2} 
 = \hate{i} \cdot \A (\A - \Id)^{-1} \hate{i},
\end{align}
where we define $\hate{i} = (\A - \Id)^2 \e{i} / \| (\A - \Id)^2 \e{i} \|$, satisfying $\| \hate{i} \|=1$. 
Analogously, using Eq.~\eqref{proof:definition:error-1}--\eqref{proof:definition:error-4}, we obtain
\begin{align}\label{proof:alpha(k+3)}
 \alpha^{(i+3)}
 = \frac{ ((\A - \Id)^2\e{i+2}) \cdot (\A(\A - \Id)\e{i+2})}{ \|(\A - \Id)^2\e{i+2}\|^2 }
 = \frac{ \hate{i} \cdot \A^3(\A - \Id)^{-1}(\A+\alpha^{(i+1)}(\Id - \A))^2 \hate{i}}{ \|\A(\A+\alpha^{(i+1)}(\Id - \A))\hate{i}\|^2 }.
\end{align}

\paragraph{Decomposition of $\hate{i}$ and useful computations}
Employing the orthogonal eigenvector basis $\{\mathbf{v}_j\}_j$, we can decompose $\hate{i} = \sum_j \beta_j \mathbf{v}_j$.
As $\| \hate{i} \| = 1$ it holds $\sum_j \beta_j^2 = 1$. By inserting the decomposition into Eq.~\eqref{proof:alpha(k+1)},
we obtain
\begin{align*}
 \alpha^{(i+1)} = \sum_j \beta_j^2 \frac{\lambda_j}{\lambda_j -1}.
\end{align*}
Hence, for the eigenvalues of $\A+\alpha^{(i+1)}(\Id - \A)$ and also the second factor of Eq.~\eqref{proof:eigenvalue:error-propagation}, it follows
\begin{align}
 \eta_j(\betaWeight)
 &:=\lambda_j + \alpha^{(i+1)} (1-\lambda_j)                                                                
 = \sum_{k \neq j} \beta_k^2 \frac{ \lambda_k - \lambda_j}{\lambda_k - 1}.  \label{proof:eigenvalue:error-propagation:second-factor}
\end{align}
Hence, for the contribution in the denominator of Eq.~\eqref{proof:alpha(k+3)}, we obtain
\begin{align*}
 \A(\A+\alpha^{(i+1)}(\Id - \A)) \hate{i}
 &= \sum_j \beta_j \lambda_j \eta_j(\betaWeight) \mathbf{v}_j.
\end{align*}
By plugging in into Eq.~\eqref{proof:alpha(k+3)} and using orthogonality of $\{\mathbf{v}_j\}_j$, we obtain for $\alpha^{(i+3)}$
\begin{align*}
 \alpha^{(i+3)} = \left[\sum_j \beta_j^2 \lambda_j^2 \eta_j(\betaWeight)^2\right]^{-1} 
 \, \left[\sum_j \beta_j^2 \frac{\lambda_j^3}{\lambda_j-1} \eta_j(\betaWeight)^2\right].
\end{align*}
By employing some arithmetics, for the third factor of Eq.~\eqref{proof:eigenvalue:error-propagation}, it follows
\begin{align}\label{proof:eigenvalue:error-propagation:third-factor}
 \lambda_j + \alpha^{(i+3)} (1-\lambda_j)
 &= \left[ \sum_k \beta_k^2 \lambda_k^2 \eta_k(\betaWeight)^2 \right]^{-1} 
 \, \left[  \sum_{k \neq j} \beta_k^2 \lambda_k^2 \frac{\lambda_k - \lambda_j}{\lambda_k - 1}\eta_k(\betaWeight)^2 \right].
\end{align}

\paragraph{Resulting eigenvalues}
By inserting Eq.~\eqref{proof:eigenvalue:error-propagation:second-factor}--\eqref{proof:eigenvalue:error-propagation:third-factor} 
into Eq.~\eqref{proof:eigenvalue:error-propagation}, we obtain for the eigenvalues of the iteration-dependent error propagation matrix 
$\A(\A + \alpha^{(i+3)} (\A - \Id)) \A(\A + \alpha^{(i+1)} (\A - \Id))$
\begin{align} \label{general-eigenvalue-bound}
 \tilde{\lambda}_j
 &= \left[ \sum_k \beta_k^2 \lambda_k^2 \eta_k(\betaWeight)^2 \right]^{-1}
 \, \left[\lambda_j^2 \eta_j(\betaWeight) \sum_{k \neq j} \beta_k^2 \lambda_k^2 \eta_k(\betaWeight)^2 \frac{\lambda_k - \lambda_j}{\lambda_k - 1}\right].
\end{align}

\paragraph{Analysis for special decomposition}
By Assumption~(C4), the initial error is spanned by two orthogonal eigenvectors. Without loss of generality
let $\e{0}\in\text{span}\{ \mathbf{v}_1,\mathbf{v}_2 \}$. Then also $\hate{i}\in\text{span}\{ \mathbf{v}_1,\mathbf{v}_2\}$
and there exist $\beta_1,\beta_2\in\mathbb{R}$ satisfying $\hate{i}=\beta_1 \mathbf{v}_1 + \beta_2 \mathbf{v}_2$ and 
$\beta_1^2+\beta_2^2=1$. Consequently, Eq.~\eqref{general-eigenvalue-bound} for $j=1$ reduces to
\begin{align*}
 \tilde{\lambda}_1
 &= \lambda_1^2\lambda_2^2 (\lambda_2 - \lambda_1)^2  \frac{(1-\gamma) \gamma }{(1-\gamma) \lambda_1^2 (\lambda_1-1)^2 + \gamma \lambda_2^2 (\lambda_2-1)^2},
\end{align*}
where $\gamma=\beta_1^2\in[0,1]$. Maximizing the second factor with respect to $\gamma\in[0,1]$ results in the upper bound
\begin{align*}
 |\tilde{\lambda}_1|
 &\leq \frac{\lambda_1^2\lambda_2^2 (\lambda_2 - \lambda_1)^2}{\left(|\lambda_1(\lambda_1-1)| + |\lambda_2(\lambda_2-1)|\right)^2} =: r(\lambda_1,\lambda_2).
\end{align*}
Due to symmetry it holds $|\tilde{\lambda}_j| \leq r(\lambda_1,\lambda_2)$, $j=1,2$. Consequently, we obtain the result.
\end{proof}

Using Lemma~\ref{lemma:anderson-star-1}, we are able to show convergence and actual acceleration of AA$^\star$($1$).

\begin{corollary}[AA$^\star$(1) accelerates contractive $\mathcal{FP}$]\label{corollary:aa:acceleration}
 Let the Assumption~(C1)--(C4) be satisfied. 
 Let $\rho(\A)<1$, where $\rho(\A)$ denotes the spectral radius of $\A$. Then it holds $r(\lambda_1,\lambda_2) < \rho(\A)^4$.
 Consequently, AA$^\star$(1) is effectively accelerating the underlying fixed-point iteration, cf.~Assumption~(C1).
\end{corollary}

\begin{proof}
By plotting $r(\lambda_1,\lambda_2)/\text{max}\left\{|\lambda_1|^4,|\lambda_2|^4\right\}$, 
we demonstrate $r(\lambda_1,\lambda_2)<\text{max}\left\{|\lambda_1|^4,|\lambda_2|^4\right\}\leq\rho(\A)$
for $(\lambda_1,\lambda_2)\in[-1,1]\times[-1,1]$ , cf.~Fig.~\ref{figure:proof:anderson-accelerates}.
\end{proof}
\begin{figure}[h!]
\centering
\subfloat[\label{figure:proof:anderson-accelerates} Acceleration rate for AA$^\star$(1) compared to theoretical contraction rate of Richardson iteration
for contractive $2\times 2$ matrices.]
{\includegraphics[width=7cm]{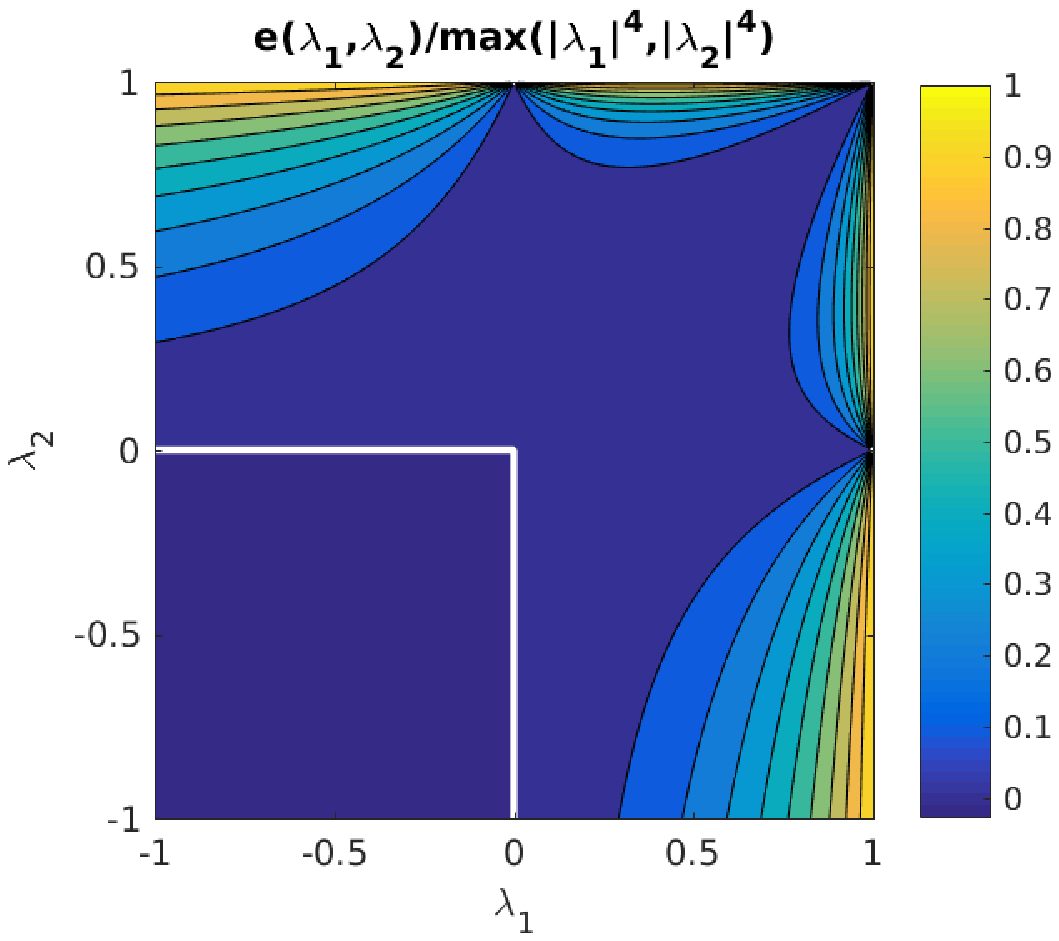}}
\hspace{0.5cm}
\subfloat[\label{figure:proof:anderson-convergence} The $(\lambda_1,\lambda_2)$-convergence plane for AA$^\star$(1) applied to $2\times2$ linear problems,
characterizing which matrices guarantee AA$^\star$ to define a uniform contraction.]
{\includegraphics[width=7cm]{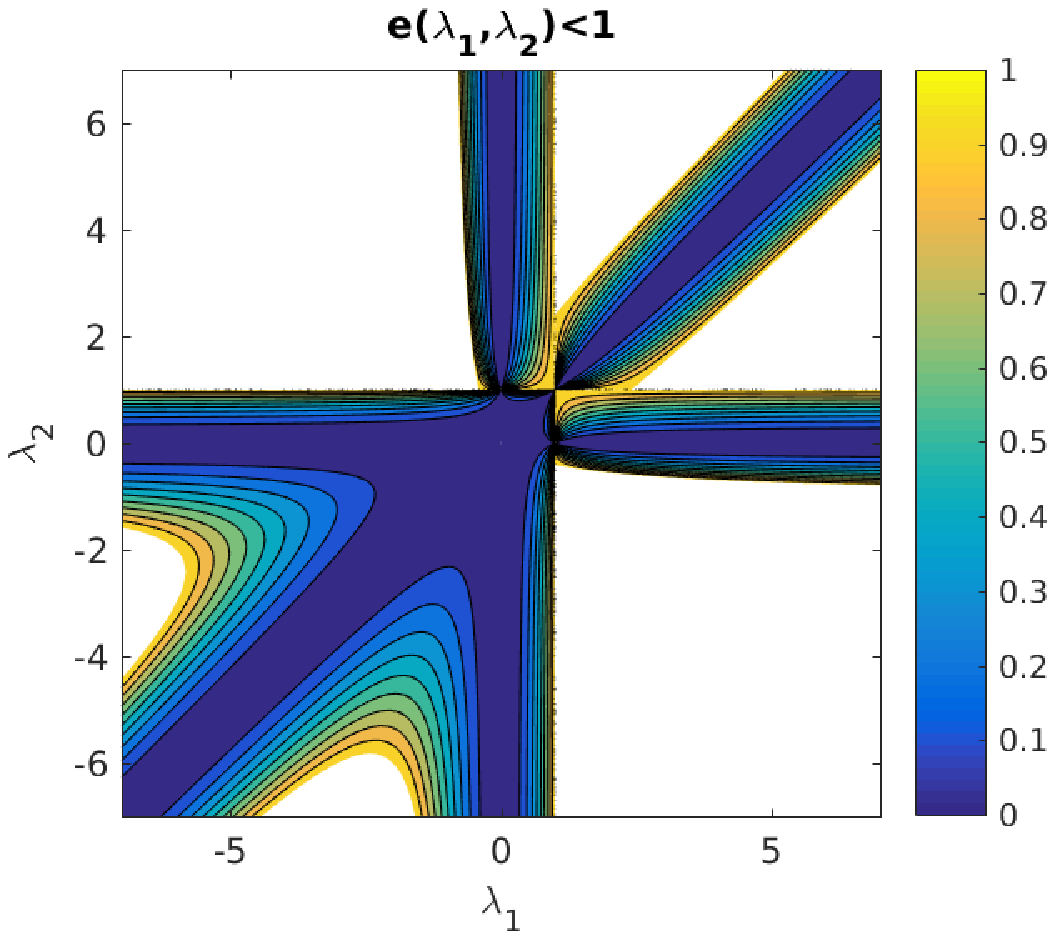}}
\caption{\label{figure:proof:planes} Acceleration and convergence factor for the restarted AA$^\star$(1).}
\end{figure}

\begin{corollary}[AA$^\star$(1) converges for non-contractive $\mathcal{FP}$]\label{corollary:aa:convergence}
 Let the Assumption~(C1)--(C4) be satisfied.  
 Let $\A$ be positive definite with at most one eigenvalue among $\{\lambda_1,\lambda_2\}$ larger than $1$ 
 and none equal to $1$. Then AA$^\star$(1) converges for the underlying non-contractive 
 fixed-point iteration, cf.~Assumption~(C1).
\end{corollary}

\begin{proof}
Due to symmetry, it is sufficient, to consider solely $(\lambda_1,\lambda_2) \in (\mathbb{R}_+\setminus \{1\}) \times (0,1)$. 
For $\lambda_1<1$, the result follows immediately from Corollary~\ref{corollary:aa:acceleration}. 
Let $\lambda_1>1$. It holds $e(1,\lambda_2)=1$ for all $\lambda_2\in(0,1)$ and $\partial_1 r(\lambda_1,\lambda_2)<0$
for all $(\lambda_1,\lambda_2)\in (1,\infty) \times (0,1)$. Thus, it follows directly that 
$r(\lambda_1,\lambda_2)<1$ for all $(\lambda_1,\lambda_2)\in (1,\infty) \times (0,1)$. 
\end{proof}

\subsection{Discussion}\label{section:linear-analysis:restarted-aa:discussion}
We make the following comments:

\begin{itemize}
  \item The convergence result in Corollary~\ref{corollary:aa:convergence} deals only with positive definite matrices.
  In Fig.~\ref{figure:proof:anderson-convergence}, eigenvalue pairs $(\lambda_1,\lambda_2)\in\mathbb{R}\times\mathbb{R}$ are displayed
  satisfying $r(\lambda_1,\lambda_2)<1$ and therefore guaranteeing AA$^\star$(1) to converge.
  In particular, AA$^\star$(1) converges also for matrices with two eigenvalues
  larger than 1 with relatively close distance to each other.
  
  \item In practice, we do not experience AA$^\star$(1) or AA(1) to fail as long as Assumption~(C4) is valid
  and $|\lambda_1|<1$ or $|\lambda_2|<1$. This observation extends also to arbitrarily large decompositions of 
  $\e{0}$ as long as at most one eigenvalue of $\A$ satisfies $|\lambda_j|>1$. 
  Based on similar observations, we state the following claim:
  If $|\lambda_j|>1$ for exactly $m$ eigenvalues $\{\lambda_j\}_j$, then AA($m$) converges
  for arbitrary $\e{0}$. We note that the worst case approach used to in order to prove 
  Lemma~\ref{lemma:anderson-star-1} cannot be applied to prove the general claim.
  It can be verified numerically that in general the eigenvalues of the error propagation 
  matrix~\eqref{general-eigenvalue-bound} can be larger than 1 even if $\rho(\A)<1$.
  
  \item From Fig.~\ref{figure:proof:anderson-convergence}, it follows, the closer the eigenvalues to $1$,
  the slower the convergence of AA$^\star$(1). This is consistent with the interpretation of Anderson acceleration 
  as secant method. The Richardson iteration does only damp slowly directions corresponding to 
  eigenvalues close to 1. Hence, a directional derivative in these directions cannot be approximated
  well purely based on the iterations of fixed-point iterations. Quite contrary to directions 
  corresponding to small or large eigenvalues relative to $1$.

  \item The theoretical convergence result has been obtained from a worst case analysis. 
  Practical convergence rates might be lower than predicted, depending on the weights of 
  the initial error.
  
  \item Convergence of AA($m$) is not guaranteed to be monotone, when applied 
  for non-contractive fixed-point iterations.
  
\end{itemize}

\section{Numerical results -- Performance study}\label{section:numerical-results}

In this section, we will show \new{three} numerical examples, \new{with increasing complexity}, comparing the linearization
schemes, presented in Section~\ref{section:linearization}, coupled with Anderson acceleration.
In particular, we confirm numerically the parabolic character of the nonlinear Biot equations, 
cf.\ Remark~\ref{remark:model:properties}, the convergence result for the Fixed-Stress-L-scheme, 
cf.\ Theorem~\ref{theorem:convergence:fsl}, as well as the acceleration and stabilization 
properties of the Anderson acceleration, cf.\ Section~\ref{section:linear-analysis:restarted-aa}.
All numerical results have been obtained using the software environment DUNE~\cite{Bastian2008,Bastian2008b,dune2.4}.

\subsection{Test case I -- Injection in a \new{2D homogeneous} medium with Lipschitz continuous constitutive laws}\label{section:numerical-results:testcase-I}

We consider a two-dimensional, homogeneous, unsaturated porous medium $(-1,1)\times(0,1)$,
in which a fluid is injected at the top, cf.\ Figure~\ref{figure:injection:geometry}. 
Due to the symmetry of the problem, we consider only the right half $\Omega=(0,1)\times(0,1)$,
discretized by $50 \times 50$ regular quadrilaterals. 
As initial condition, we choose a constant displacement and pressure field with $\u(0)=\bm{0}$ and 
$\p(0)=p_0$, satisfying the stationary version of the continuous problem~\eqref{model:p}--\eqref{model:u}.
In order to avoid inconsistent initial data, we ramp the injection at the top with inflow rate 
$q_\mathrm{inflow}(t) = q^\star \times \text{min}\,\{t^2, 1.0\}$ for given $q^\star\in\mathbb{R}$.
Apart from the inflow at the top, we consider no flow at the remaining boundaries, no normal 
displacement at left, right and bottom boundary and no stress on the top. The boundary conditions
are displayed in Figure~\ref{figure:injection:geometry}.

\begin{figure}[h!]
\begin{center}
\vspace{0.5cm}
 \begin{overpic}[width=9cm]{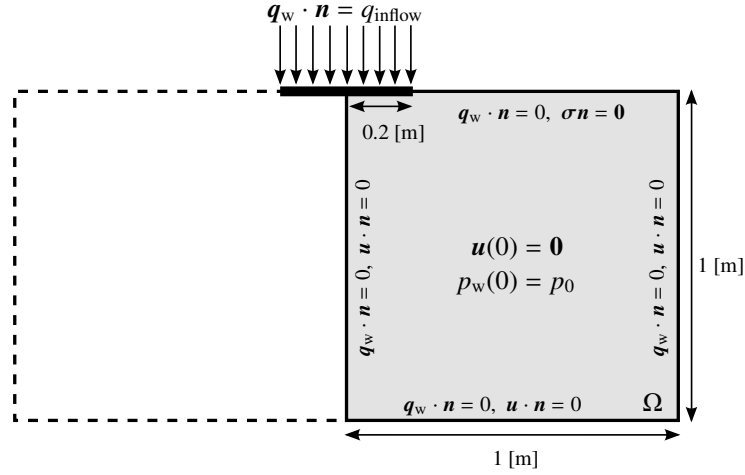}
  \put(67,28){$\u(0)=\bm{0}$}
  \put(64.75,23){$\p(0)=p_0$}
  \put(37,63){$\q\cdot\n = q_\mathrm{inflow}$}
  \put(65,48){\footnotesize $\q\cdot\n = 0$,\ \ $\stress\n = \bm{0}$}
  \put(93,13){\rotatebox{90}{\footnotesize $\q\cdot\n = 0$,\ \  $\u\cdot\n = 0$}}
  \put(50,13){\rotatebox{90}{\footnotesize $\q\cdot\n = 0$,\ \  $\u\cdot\n = 0$}}
  \put(57,5){\footnotesize $\q\cdot\n = 0$,\ \  $\u\cdot\n = 0$}
  \put(70,-3){\footnotesize $1\ [\mathrm{m}]$}
  \put(100,25.5){\footnotesize $1\ [\mathrm{m}]$}
  \put(51,45){\footnotesize $0.2\ [\mathrm{m}]$}
  \put(92,5){$\Omega$}
 \end{overpic}
\end{center}
 \caption{Domain $\Omega$ and boundary conditions for test case~I.}
 \label{figure:injection:geometry}
\end{figure}

\paragraph{Physical and numerical parameters}
For the constitutive laws, governing saturation and permeability, we use the van Genuchten-Mualem 
model~\cite{vanGenuchten1980}, defining
\begin{align*}
  \s(\pw) 	&= \left\{ \begin{array}{l l} (1 + (-\aVG \pw)^\nVG)^{-\frac{\nVG-1}{\nVG}} &, \pw\leq 0, \\ 1&,\text{else,}\end{array} \right.
  &
  \permeability(\s) 	&= \frac{\k}{\mu_w} \,\sqrt{\s} \, \left(1-\left(1-\s^\frac{\nVG}{\nVG-1}\right)^\frac{\nVG-1}{\nVG}\right)^2,\ \s\in[0,1],
\end{align*}
where $\aVG$ and $\nVG$ are model parameters associated to the inverse of the air suction value and pore size 
distribution, respectively, $\k$ is the intrinsic absolute permeability and $\mu_w$ is the dynamic fluid viscosity.

\begin{table}[h!]
\centering
\def\arraystretch{1.1}
{\footnotesize
 \begin{tabular}{|l|l|l|l|l|}
 \hline
 Parameter  & Variable & Test case I   & Test case II & Test case III \\
            & [unit]   & (Section~\ref{section:numerical-results:testcase-I}) & (Section~\ref{section:numerical-results:testcase-II}) & (Section~\ref{section:numerical-results:testcase-III}) \\
 \hline
 \hline
 Young's modulus		& $E$ [Pa]			& 3e1    & \new{3e1}	 			& 1e6              \\
 Poisson's ratio		& $\nu$	[-]			& 0.2    & 0.2		 			& 0.3              \\
 Initial pressure 		& $p_0$ [Pa]			& -7.78  & $-$15.3	 			& \textit{hydrostatic} \\
 Initial porosity		& $\phi_0$ [-]			& 0.2    & \new{0.2}	 			& 0.2              \\
 Inverse of air suction		& $\aVG$ [Pa$^{-1}$]		& 0.1844 & 0.627	 			& $1\mathrm{e-}4$  \\
 Pore size distribution		& $\nVG$ [-]			& 3.0    & 1.4		 			& 0.7$^{-1}$       \\
 Abs. permeability		& $\k$ [m$^2$]			& 3e$-$2 & \new{3e$-$2}	 			& $5\mathrm{e-}13$ \\
 Fluid viscosity		& $\mu_w$ [Pa$\cdot$s]		& 1.0    & \new{1.0}  	 			& $1\mathrm{e-}3$  \\
 Gravitational acc.     	& $\g$ [m/s$^2$]	        & 0.0    & 0.0	 				& 9.81             \\
 Biot coefficient		& $\alpha$ [-]			& 0.1 $|$ 0.5 $|$ 1.0   & 0.1 $|$ 0.5 $|$ 1.0  	& 1.0              \\
 Biot modulus			& $N$ [Pa]			& $\infty$& $\infty$	 			& $\infty$         \\
 Maximal inflow rate		& $q^\star$ [m$^2$/s]		& -1.25 & -0.175  				& [-] \\
 Final time			& $T\ [\mathrm{s}]$		& 1.0	 & 1.0		 			& 86400\ ($=10\ [\mathrm{days}]$) \\
 \hline
 Time step size	                & $\tau\ [\mathrm{s}]$     	& 1e$-$1 & 1e$-$1  				& 3600\ ($=1\ [\mathrm{hours}]$)  \\
 Absolute tolerance		& $\varepsilon_\mathrm{a}$	& 1e$-$8 & 1e$-$8 	 			& $1\mathrm{e-}3$ \\
 Relative tolerance		& $\varepsilon_\mathrm{r}$	& 1e$-$8 & 1e$-$8  				& $1\mathrm{e-}6$ \\
 \hline
 \end{tabular}
 }
 \caption{\label{table:numerical-results:parameters} Parameters employed for test cases I and II. 
 Top: Physical model parameters. Bottom: Numerical parameters.}
\end{table}

Values chosen for model parameters and numerical parameters are displayed in 
Table~\ref{table:numerical-results:parameters}. The parameters have been chosen such 
that the initial saturation is $\szero=0.4$ in $\Omega$ and a region of full saturation 
($\s=1$) is developed after seven time steps. Furthermore, the constitutive laws for 
saturation and permeability are Lipschitz continuous (with $\LipschitzS=0.12$). We 
consider three different values for the Biot coefficient, controlling whether the 
Richards equation or the nonlinear coupling terms determine the character of the 
numerical difficulties. The simulation result for strong coupling ($\alpha=1.0$) at final 
time $t=1$ is illustrated exemplarily in Figure~\ref{figure:injection:lipschitz-final-state}.

\begin{figure}[h!]
\begin{center}
\begin{overpic}[width=1.05\textwidth]{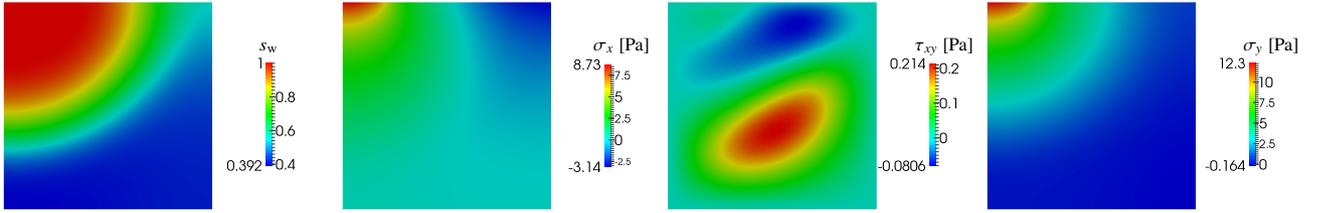}
 \put(20,13.4){\new{{\scriptsize $\s$}}}
 \put(45.5,13.4){\new{{\scriptsize $\sigma_x$ [Pa]}}}
 \put(70,13.4){\new{{\scriptsize $\tau_{xy}$ [Pa]}}}
 \put(95,13.4){\new{{\scriptsize $\sigma_y$ [Pa]}}}
\end{overpic}
 \caption{\label{figure:injection:lipschitz-final-state}Simulation results for test case~I (Lipschitz continuous permeability): 
 \new{Saturation, normal stresses $\sigma_x$, $\sigma_y$ and shear stress $\tau_{xy}$ at time $t=1$.}}
\end{center}
\end{figure}

\paragraph{Performance of linearization schemes}

We consider the four linearization schemes introduced in Section~\ref{section:linearization},
coupled with Anderson acceleration as post-processing. Abbreviations used in this section are 
introduced in Table~\ref{table:abbreviations-methods}.

\begin{table}[h!]
 \centering
\setlength{\tabcolsep}{0.4em}
\def\arraystretch{1.1}
\begin{tabular}{|l|l|}
 \hline
 Abbreviation 			& Explanation \\
 \hline \hline 
 Newton				& Monolithic Newton's method \\
 FS-Newton			& Fixed-Stress-Newton method \\
 FS-MP				& Fixed-Stress-Modified-Picard method \\
 FSL	 			& Fixed-Stress-L-scheme with $L=\LipschitzS + \beta_\mathrm{FS}$ \\
 FSL/2	 			& Fixed-Stress-L-scheme with $L = \tfrac{1}{2} (\LipschitzS + \beta_\mathrm{FS})$ \\
 \hline
 \texttt{LIN}-AA($m$)		& Anderson accelerated linearization scheme \texttt{LIN} with $m+1$ reused iterations \\
 \hline
\end{tabular}
\caption{Abbreviations for methods (top) and additional stabilizations (bottom), where \texttt{LIN} is a variable.}
\label{table:abbreviations-methods}
\end{table}

We use the average number of iterations per time step as measure for performance,
cf.\ Table~\ref{table:test-case-I:performance:1-3}. In particular, we disregard the use of CPU time as performance measure
due to a not finely-tuned implementation. We just note, that a single iteration of a splitting method is significantly
faster than a single monolithic Newton iteration.

First of all, all plain linearization schemes (AA(0)) succeed to converge for all three coupling 
strengths. This is consistent with Remark~\ref{remark:model:properties},
demonstrating that the nonlinear Biot equations do not adopt the degeneracy of the Richards equation
and remain parabolic in a fully saturated regime.
Not surprisingly, the monolithic Newton method requires fewest iterations. So at first impression,
it seems to be the preferred method. However, as stressed above, an advanced monolithic solver or an 
fixed-stress type iterative solver are required for efficient solution independent of the coupling 
strength,  i.e., additional costs are hidden. 
On the other hand, the remaining linearization schemes allow separate simulators from the beginning. 
As the Fixed-Stress-L-scheme does not utilize an exact evaluation of derivatives, the Fixed-Stress-Newton 
method and the Fixed-Stress-Modified-Picard perform better for all three coupling strengths. 
Solely the performance of the Fixed-Stress-Newton method shows weak dependence on the coupling strength.
The remaining methods show improved convergence behavior for increasing coupling strength,
due to the decreasing numerical complexity of the problem itself following from Remark~\ref{remark:model:properties}.

\begin{table}[h!]
\centering
\def\arraystretch{1.3}
\setlength{\tabcolsep}{0.4em}
{\small
 \begin{tabular}{llccclccclccclccclccc}
 \hline
  Linearization & & \multicolumn{3}{c}{Newton} & & \multicolumn{3}{c}{FS-Newton}& & \multicolumn{3}{c}{FS-MP}& & \multicolumn{3}{c}{FSL}& & \multicolumn{3}{c}{FSL/2} \\
  \multicolumn{1}{l}{Biot coeff. $\alpha$} & & 0.1 & 0.5 & 1.0 & & 0.1 & 0.5 & 1.0 & & 0.1 & 0.5 & 1.0 & & 0.1 & 0.5 & 1.0 & & 0.1 & 0.5 & 1.0 \\
  \cline{1-1} \cline{3-5} \cline{7-9} \cline{11-13} \cline{15-17} \cline{19-21}
  AA(0)   &   & \textbf{5.3} & \textbf{5.1} & \textbf{5.0}& 
              & \textbf{6.0} & 8.3 & 10.6& 
              & 18.2 & 18.2 & 16.7 & 
              & 23.2 & 21.2 & 18.9 & 
              & 46.8 & 41.4 & 41.1 
              \\
  AA(1)   &   & 6.1 & 6.0 & 6.0& 
              & 6.2 & \textbf{7.6} & 8.9& 
              & 15.8 & 15.5 & 15.7 & 
              & 21.2 & 19.7 & 17.7 & 
              & 17.4 & 17.3 & 17.3 
              \\
  AA(3)   &   & 7.4 & 7.4 & 7.5& 
              & 7.4 & 7.7 & 8.5 & 
              & 13.4 & 13.6 & 13.5 &
              & 16.1 & 15.3 & 15.0 & 
              & 14.3 & 14.5 & 14.7  
              \\
  AA(5)   &   & 8.3 & 8.1 & 8.2 & 
              & 7.9 & 7.9 & \textbf{8.4} & 
              & 13.1 & 12.8 & 12.5 & 
              & 14.9 & 14.6 & 14.3 & 
              & 13.3 & 13.5 & 13.6 
              \\
  AA(10)  &   & -- & -- & --& 
              & -- & -- & --& 
              & \textbf{12.8} & \textbf{12.5} & \textbf{12.3}& 
              & \textbf{14.4} & \textbf{14.3} & \textbf{14.1} & 
              & \textbf{13.3} & \textbf{13.1} & \textbf{13.4}  
              \\
  \hline
 \end{tabular}
 }
  \caption{\label{table:test-case-I:performance:1-3} Performance for test case I with different coupling strengths 
  ($\alpha=0.1,\ 0.5,\ 1.0$). \new{Average number of (nonlinear) iterations per time step for Newton's method, the 
  Fixed-Stress-Newton method, the Fixed-Stress-Modified-Picard method and the Fixed-Stress-L-scheme; both plain and 
  coupled with Anderson acceleration for different depths ($m=1,\ 3,\ 5,\, 10$).}
  Minimal numbers per linearization type and Biot coefficient are in bold.}
\end{table}

When applying Anderson acceleration, we observe that Anderson acceleration slows down 
the convergence of the monolithic Newton method, which is consistent with considerations in 
Section~\ref{section:acceleration}. In contrast, Anderson acceleration speeds up significantly the convergence 
of the Picard-type methods (Fixed-Stress-L-scheme and the variation FSL/2, and Fixed-Stress-Modified-Picard).
Largest acceleration effect can be seen for largest considered depth.
For the Fixed-Stress-Newton method, the effect of Anderson acceleration depends on the numerical character 
of the problem. This is due to the fact, that for weak coupling, the method is essentially identical with 
Newton's method.

Regarding the Fixed-Stress-L-scheme, according to Theorem~\ref{theorem:convergence:fsl}, optimally the 
diagonal stabilization parameter has to be chosen as small as possible. However, smaller 
values do not necessarily lead to faster convergence, as can be observed by comparing the plain
Fixed-Stress-L-scheme and the plain FSL/2-scheme. Yet when utilizing Anderson acceleration, 
robustness with respect to the tuning parameter is increased, and eventually the FSL/2-scheme 
converges faster than the Fixed-Stress-L-scheme. In particular, it performs as good as the 
Fixed-Stress-Modified-Picard method.

All in all, the theory has been confirmed. The Fixed-Stress-L-scheme converges despite the simple 
linearization approach and Anderson acceleration is able to accelerate Picard-type schemes.
Moreover, the latter has been shown to stabilize the Fixed-Stress-L-scheme, allowing to choose
a small tuning parameter leading to improved convergence behavior.
Considering the cost per iteration, despite some additional iterations, we finally recommend the use of
the Fixed-Stress-Newton method with Anderson acceleration with low depth. It is cheap and allows 
separate simulators. For strongly coupled problems or in the absence of exact derivatives, 
the Fixed-Stress-L-scheme with small tuning parameter is an attractive alternative to the 
Fixed-Stress-Newton method.

\subsection{Test case II -- Injection in \new{2D homogeneous} medium with H\"older continuous permeability}\label{section:numerical-results:testcase-II}

In the following, we reveal the limitations of the considered linearization schemes. 
Moreover, we demonstrate the stabilization property of Anderson acceleration, allowing 
non-convergent methods to converge. For this purpose, we repeat test case~I
with modified physical parameters.  In particular, we choose the saturation to be 
Lipschitz continuous with same Lipschitz constant as in test case~I. In contrast, 
the permeability is chosen to be only H\"older continuous. Hence, the derivative
becomes unbounded in the transition between partial and full saturation, causing 
potential trouble for the Newton-type methods. Again, we choose the initial pressure 
and the maximal inflow rate such that $\szero=0.4$ and a region of full saturation ($\s=1$)
is developed after seven time steps. The simulation result for strong coupling at final 
time $t=1$ is illustrated in Figure~\ref{figure:injection:hoelder-final-state}.

\begin{figure}[h!]
\begin{center}
\begin{overpic}[width=1.05\textwidth]{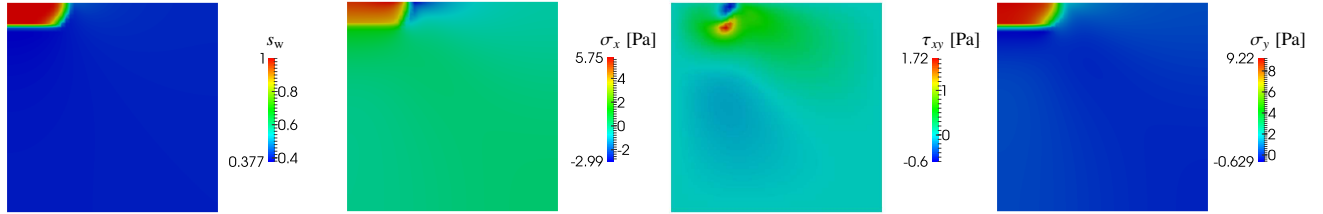}
 \put(20,13.4){\new{{\scriptsize $\s$}}}
 \put(45.5,13.4){\new{{\scriptsize $\sigma_x$ [Pa]}}}
 \put(70,13.4){\new{{\scriptsize $\tau_{xy}$ [Pa]}}}
 \put(95,13.4){\new{{\scriptsize $\sigma_y$ [Pa]}}}
\end{overpic}
 \caption{\label{figure:injection:hoelder-final-state}Simulation results for test case~II (H\"older continuous permeability): 
 \new{Saturation, normal stresses $\sigma_x$, $\sigma_y$ and shear stress $\tau_{xy}$ at time $t=1$.}}
\end{center}
\end{figure}

As mentioned in Remark~\ref{remark:limits-l}, due to lack of regularity for the permeability, 
each of the considered methods faces difficulties. For Newton-type methods (Newton, Fixed-Stress-Newton), 
the derivative of the permeability is evaluated, which might be unbounded.  Effectively, for 
the Fixed-Stress-L-scheme, this also means that $\LipschitzK\rightarrow\infty$ or in practice 
$\LipschitzK$ becomes very large. Hence, by Theorem~\ref{theorem:convergence:fsl}, the time step
size has to be chosen sufficiently small or possibly $L$ has to be chosen larger to guarantee convergence. 
We note that for chosen initial saturation the permeability is significantly lower than for test case~I. 
Consequently, the theoretical convergence rate for the plain Fixed-Stress-L-scheme~\eqref{fsl:p}--\eqref{fsl:u}
deteriorates. Due to round off errors stagnation is possible.

\paragraph{Performance of linearization schemes}

\begin{table}[h!]
\centering
\def\arraystretch{1.3}
\setlength{\tabcolsep}{0.3em}
{\small
 \begin{tabular}{llccclccclccclccclccc}
 \hline
  Linearization & & \multicolumn{3}{c}{Newton} & & \multicolumn{3}{c}{FS-Newton}& & \multicolumn{3}{c}{FS-MP}& & \multicolumn{3}{c}{FSL}& & \multicolumn{3}{c}{FSL/2} \\
  \multicolumn{1}{l}{Biot coeff. $\alpha$} & & 0.1 & 0.5 & 1.0 & & 0.1 & 0.5 & 1.0 & & 0.1 & 0.5 & 1.0 & & 0.1 & 0.5 & 1.0 & & 0.1 & 0.5 & 1.0 \\
  \cline{1-1} \cline{3-5} \cline{7-9} \cline{11-13} \cline{15-17} \cline{19-21}
  AA(0)   &   & $\bm{\nearrow}[8]$ & \textbf{8.5} & \textbf{8.1}& 
              & $\bm{\nearrow}[9]$ & 13.2 & 19.1& 
              & $\bm{\rightarrow}[\new{3}]$ & 36.9 & 55.0 & 
              & $\bm{\rightarrow}[9]$ & 126.9 & 134.9 & 
              & $\bm{\rightarrow}[8]$ & $\bm{\rightarrow}[9]$ & $\bm{\rightarrow}[10]$ 
              \\
  AA(1)   &   & \textbf{10.7} & 9.4 & $\bm{\rightarrow}[8]$& 
              & \textbf{11.0} & \textbf{11.8} & 14.6& 
              & 45.2 & 34.2 & 33.8 & 
              & 133.6 & 84.0 & 83.2 & 
              & $\bm{\rightarrow}[9]$ & 68.5 & 65.1 
              \\
  AA(3)   &   & 17.2 & 11.7 & $\bm{\rightarrow}[8]$ & 
              & 15.6 & 12.1 & \textbf{13.0} & 
              & 30.5 & 26.9 & 28.1 &
              & 68.3 & 54.3 & 56.9 & 
              & 48.4 & 37.9 & 35.5  
              \\
  AA(5)   &   & 24.8 & 13.9 & $\bm{\rightarrow}[8]$ & 
              & 23.3 & 13.1 & 13.2 & 
              & 29.2 & 24.7 & 23.5 & 
              & 62.4 & 48.7 & 44.9 & 
              & 43.4 & 34.8 & 32.7 
              \\
  AA(10)  &   & 33.3 & 18.4 & $\bm{\rightarrow}[8]$& 
              & 43.0 & 14.7 & 13.8 & 
              & \textbf{29.8} & \textbf{23.5} & \textbf{23.5}& 
              & \textbf{52.6} & \textbf{42.6} & \textbf{42.5} & 
              & \textbf{39.3} & \textbf{31.8} & \textbf{29.2}  
              \\
  \hline
 \end{tabular}
 }
  \caption{\label{table:test-case-II:performance:1-4} Performance for test case II with different coupling strengths 
  ($\alpha=0.1,\ 0.5,\ 1.0$). \new{Average number of (nonlinear) iterations per time step for Newton's method, 
  the Fixed-Stress-Newton method, the Fixed-Stress-Modified-Picard method and the Fixed-Stress-L-scheme; 
  both plain and coupled with Anderson acceleration for different depths ($m=1,\ 3,\ 5,\, 10$).} 
  Minimal numbers per linearization type and Biot coefficient are in bold. Failing linearization due to 
  stagnation at time step $n$ is marked by $\bm{\rightarrow}[n]$. Failing linearization due to divergence 
  at time step $n$ is marked by $\bm{\nearrow}[n]$.}
\end{table}

The average number of iterations per time step is presented in Table~\ref{table:test-case-II:performance:1-4}.
In contrast to test case~I, not all plain linearization schemes (AA(0)) converge. For weak coupling, 
all Newton-like methods (Newton, Fixed-Stress-Newton) diverge with the Fixed-Stress-Newton method
being slightly more robust due to added fixed-stress stabilization. The Fixed-Stress-L-scheme
stagnates and shows to be slightly more robustness than the Newton-type methods. 
The Fixed-Stress-Modified-Picard method is least robust and stagnates already after two time steps.
For strong coupling, all methods converge, which is consistent with Remark~\ref{remark:model:properties}.
If convergent, the schemes sorted by required number of iterations are the monolithic 
Newton method, the Fixed-Stress-Newton method, the Fixed-Stress-Modified-Picard and the Fixed-Stress-L-scheme,
meeting our expectations.

By utilizing Anderson acceleration, convergence can be observed for all coupling 
strengths and all linearization schemes besides Newton's method for $\alpha=1$.
In particular, all previously failing schemes converge. 
This confirms the possible increase of robustness by Anderson 
acceleration, postulated in Section~\ref{section:linear-analysis:restarted-aa}.
Similar observations as before are made for the splitting schemes under Anderson acceleration.
All in all, the theory has been confirmed.

As before, for increasing depth, the performance of Newton's method deteriorates.
For strong coupling stagnation is observed. For weak coupling, for several time steps
practical stagnation is observed with eventual convergence after a very large number
of iterations. This is consistent with the fact that Anderson 
acceleration can also lead to divergence for increasing depth~\cite{Walker2011}.
Hence, Anderson acceleration has to be applied carefully for the monolithic Newton method.

Motivated by test case~I, we apply the Fixed-Stress-L-scheme with a decreased tuning parameter.
For this test case, the plain FSL/2-scheme fails for all coupling strengths.
As the FSL/2-scheme is \textit{a priori} less robust as the Fixed-Stress-L-scheme, 
this has been expected. Utilizing Anderson acceleration, the FSL/2-scheme eventually converges.
In particular, convergence is always faster than for the corresponding Fixed-Stress-L-scheme.
This again demonstrates the ability of the Anderson acceleration to increase robustness and
to relax assumptions for practical convergence.

According to the theory for the Fixed-Stress-L-scheme, a larger tuning parameter or a lower time 
step size could enable convergence, e.g., for $\alpha=0.1$, AA(0). However, we do not consider
those strategies here, as they lead to worse convergence rates and utilizing Anderson 
acceleration should be anyhow preferred.

Concerning the best splitting method, we again recommend the use of the Fixed-Stress-Newton 
method combined with Anderson acceleration with low depth. It is cheap, robust and allows 
separate simulators.

\subsection{Test case III -- Unsteady seepage flow through a 2D homogeneous levee}\label{section:numerical-results:testcase-III}

We consider unsteady seepage flow through a simple, two-dimensional, homogeneous levee, enforced by a flood. The levee consists of a lower and upper part (lower $5\ [\mathrm{m}]$ and upper $10\ [\mathrm{m}]$, respectively), cf.\ Figure~\ref{figure:test-case-III:domain}.  Initially, the water table lies at the interface between lower and upper part. The initial fluid pressure is a hydrostatic pressure with $p=0$ at the water table. The reference configuration, defined by the domain, is initially already consolidated under the influence of gravity. As $\u$ is the deviation of the reference configuration, effectively, no gravity is applied in the mechanics equation, but only in the flow equation. 

\begin{figure}[h!]
\centering
\begin{overpic}[width=12cm]{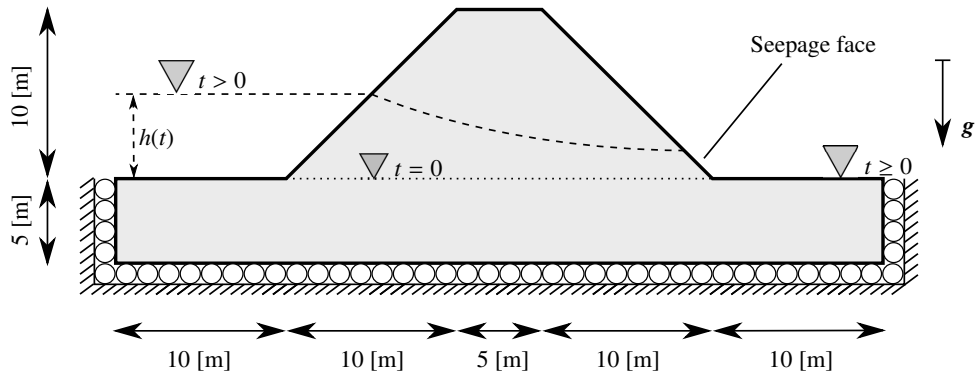}
  \put(78,32){\small Seepage face}
  \put(17.5,28){\small $t>0$}
  \put(39,18.3){\small $t=0$}
  \put(90.5,18.5){\small $t\geq0$}
  \put(101,23){\small $\bm{g}$}
  \put(11,21.7){\small $h(t)$}
  \put(-3,23.5){\rotatebox{90}{\small $10\ [\mathrm{m}]$}}
  \put(-3,10.5){\rotatebox{90}{\small $5\ [\mathrm{m}]$}}
  \put(14,-2.8){\small $10\ [\mathrm{m}]$}
  \put(33,-2.8){\small $10\ [\mathrm{m}]$}
  \put(48,-2.8){\small $5\ [\mathrm{m}]$}
  \put(61,-2.8){\small $10\ [\mathrm{m}]$}
  \put(80,-2.8){\small $10\ [\mathrm{m}]$}
\end{overpic}
\vspace{0.3cm}
\caption{\label{figure:test-case-III:domain} Domain, boundary and initial conditions for test case~III.}
\end{figure}

Over time, on the left hand side of the levee, the water table rises with constant speed for four days and remains constant for the next six days, defining $h(t)=2t\ [\mathrm{m/days}]$, $t\leq 4\ [\mathrm{days}]$ and $h(t)=8\ [\mathrm{m}]$, $t\geq 4\ [\mathrm{days}]$. Below $h(t)$ on the left, a hydrostatic pressure boundary condition is applied. On the right side, we apply approximate seepage face boundary conditions, based on the previous time step; i.e., given a fully saturated cell at the previous time step, a pressure boundary condition $p=0$ is applied on corresponding boundary for the next time step, otherwise a no-flow boundary condition is applied for the volumetric flux. On the remaining boundary, no-flow boundary conditions are applied for all time. For the mechanics, no displacement in normal direction is assumed on the boundary of the lower part of the levee. On the boundary of the upper part and the interface, zero effective stress is applied. The boundary conditions are visualized in Figure~\ref{figure:test-case-III:domain}.

\paragraph{Physical and numerical parameters}
The domain is discretized by a regular, unstructured, simplicial mesh with approximately 67,000 elements and 201,000 nodes. Compared to the previous test cases, we employ more realistic material parameters. Values chosen for model parameters and numerical parameters are displayed in Table~\ref{table:numerical-results:parameters}. We note, the resulting permeability is only H\"older continuous. The saturation history and deformation at four times is displayed in Figure~\ref{figure:test-case-III:saturation-history}. We observe steep saturation gradients during the flooding. Furthermore, both consolidation and swelling can be observed. All in all, the levee is pushed to the right.

\begin{figure}[h!]
 \centering
 \begin{overpic}[width=12cm]{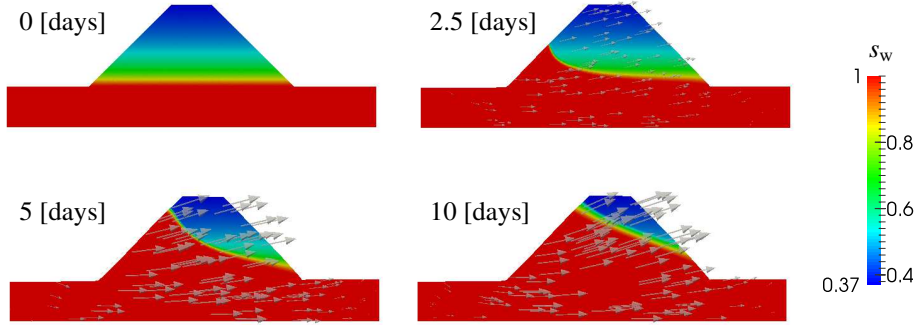}
  \put(2,35){$0\ [\mathrm{days}]$}
  \put(47,35){$2.5\ [\mathrm{days}]$}
  \put(2,14){$5\ [\mathrm{days}]$}
  \put(47,14){$10\ [\mathrm{days}]$}
  \put(95,32){$\s$}
 \end{overpic}
 \caption{\label{figure:test-case-III:saturation-history} Simulation result for test case III. Saturation for the deformed material at time $t=0\ [\mathrm{days}], 2.5\ [\mathrm{days}], 5\ [\mathrm{days}], 10\ [\mathrm{days}]$. Direction and intensity of deformation also indicated by arrows.}
\end{figure}

\paragraph{Performance of linearization schemes}

We consider the same linearization schemes as in the previous test cases, all but FSL, i.e., the Fixed-Stress-L-scheme with $L=\LipschitzS + \beta_\mathrm{FS}$. Based on previous observations, we expect FSL/2 coupled with Anderson acceleration to be more efficient than FSL. The average number of iterations per time step is presented in Table~\ref{table:test-case-III:performance}. First of all, we observe that all plain linearization schemes fail in the same phase of the simulation (after around 50 time steps). The reason for that lies mainly in the steep saturation gradients. As before, Anderson acceleration can yield remedy. However, for this test case, the simple combination of Newton's method and Anderson acceleration does not converge for any considered depth. Newton's method combined with Anderson acceleration is still not convergent for AA(1). For increasing depth the robustness decreases again, which is consistent with observations from the previous test cases. For the remaining linearization schemes convergence can be obtained. In particular, the Fixed-Stress-Newton method combined with AA(1) converges with the least amount of iterations. The Picard-type methods are slower, but show again more robustness with respect to increasing depth, whereas the Fixed-Stress-Newton method diverges eventually for $m=10$. Here, the Picard type methods require at least depth $m=3$ for successful convergence. After all, we conclude that the diagonal stabilization is essential for the success of the linearization schemes. The stabilization is added via both the fixed-stress splitting scheme and the L-scheme. Consequently, we expect also the monolithic Newton method to be convergent when adding sufficient diagonal stabilization.

\begin{table}[h!]
\centering
\def\arraystretch{1.3}
\setlength{\tabcolsep}{0.3em}
{\small
 \begin{tabular}{llclclclc}
 \hline
  Linearization & & Newton & & FS-Newton& & FS-MP& & FSL/2 \\
  \hline
  AA(0)   &   & $\bm{\nearrow}[56]$ & 
              & $\bm{\rightarrow}[57]$ & 
              & $\bm{\rightarrow}[48]$ & 
              & $\bm{\rightarrow}[48]$ 
              \\
  AA(1)   &   &  $\bm{\rightarrow}\textbf{[165]}$ & 
              & \textbf{10.2}& 
              & $\bm{\rightarrow}[86]$ & 
              & $\bm{\rightarrow}[73]$  
              \\
  AA(3)   &   & $\bm{\rightarrow}[90]$& 
              & 11.4 & 
              & 18.1 & 
              & 33.2  
              \\
  AA(5)   &   & $\bm{\rightarrow}[87]$& 
              & 10.6 & 
              & 16.9 & 
              & 30.2  
              \\
  AA(10)  &   & $\bm{\rightarrow}[87]$ & 
              & $\bm{\rightarrow}[87]$ & 
              & \textbf{16.0} & 
              & \textbf{28.3}  
              \\
  \hline
 \end{tabular}
 }
  \caption{\label{table:test-case-III:performance} Performance for test case III. Average number of (nonlinear)
  iterations per time step for Newton's method, the Fixed-Stress-Newton method, the Fixed-Stress-Modified-Picard 
  method and the Fixed-Stress-L-scheme; both plain and coupled with Anderson acceleration for different depths ($m=1,\ 3,\ 5,\, 10$). 
  Minimal numbers per linearization type are in bold. Failing linearization due to stagnation at time step $n$ is marked
  by $\bm{\rightarrow}[n]$. Failing linearization due to divergence at time step $n$ is marked by $\bm{\nearrow}[n]$.}
\end{table}

\section{Concluding remarks}\label{section:concluding-remarks}

In this paper, we have proposed three different linearization schemes for nonlinear
poromechanics of unsaturated materials. All schemes incorporate the fixed-stress 
splitting scheme and allow the efficient and robust decoupling of mechanics and 
flow equations. In particular, the simplest scheme, the Fixed-Stress-L-scheme, 
employs solely constant diagonal stabilization. It has been derived as L-scheme 
linearization of the Biot equations reduced to a pure pressure formulation. 
Under mild, physical assumptions, also needed for the mathematical model to be valid,
it has been rigorously shown to be a contraction. This also has been verified 
numerically. Exploiting the derivation of the Fixed-Stress-L-scheme allows 
modifications including first order Taylor approximations. In this way, we have 
introduced the Fixed-Stress-Modified-Picard and the Fixed-Stress-Newton method.

The derivation of the Fixed-Stress-L-scheme provides two particular side products.
First, it reveals the close relation of the L-scheme and the fixed-stress splitting
scheme. Second, the nonlinear Biot equations can be shown to be parabolic in the 
pressure variable. This holds in particular in the fully saturated regime
unlike for the Richards equation.

The theoretical convergence rate of the Fixed-Stress-L-scheme might deteriorate for
unfavorable situations, leading to slow convergence or even stagnation in practice. 
Similarly, the Fixed-Stress-Modified-Picard and Fixed-Stress-Newton methods are 
prone to diverge for H\"older continuous nonlinearities. In order to accelerate 
or retain convergence, we apply Anderson acceleration, which is a post-processing,
maintaining the decoupled character of the underlying splitting methods. The general 
increase of robustness and acceleration of convergence via the Anderson acceleration 
has been justified theoretically considering a special linear case. To our knowledge,
this is the first theoretical indication of this kind, considering non-contractive
fixed-point iterations.

In practice, Anderson acceleration has shown to be very effective for the considered Picard-type
methods, confirming the theoretical considerations. After all, we recommend the 
combination of the Fixed-Stress-Newton method and the Anderson acceleration, being 
very robust even for H\"older continuous nonlinearities. In case analytical 
derivatives are not available, we recommend the combination of the Fixed-Stress-L-scheme
with a decreased tuning parameter and Anderson acceleration. Without Anderson 
acceleration, convergence might not be guaranteed. Including it, does not only retain 
but it also significantly accelerates convergence. This is interesting, as the optimal
tuning parameter is not necessarily known \textit{a priori} and can be more safely approached
under the use of Anderson acceleration.

As outlook, with focus on large scale applications, the performance of the linearization
schemes should be analyzed under the use of parallel, iterative solvers; in particular,
as due to added stabilization, the arising linear systems are expected to be better 
conditioned than for the monolithic Newton method. Additionally, Anderson acceleration
should be further studied in the context of possibly non-contractive fixed point iterations.
Examples are (i) the linearization of degenerate problems including 
H\"older continuities, which are known to be difficult to solve~\cite{Both2018b}, and 
(ii) numerical schemes employing a tuning parameter. Based 
on the numerical results in this paper, the approach seems very promising.

\section*{Acknowledgments}

This work was partially supported by the Norwegian Academy of Science and Letters and Statoil
through the VISTA AdaSim project \#6367 and by the Norwegian Academy of Science and Letters 
through NFR project 250223.

\appendix

\section{Convergence proof of abstract L-scheme}\label{section:appendix}
We present the proof of Lemma~\ref{lemma:abstract:l-scheme:convergence} showing convergence of the 
L-scheme~\eqref{algebraic:l-scheme:abstract} as linearization for Eq.~\eqref{algebraic:nonlinear:abstract}.
The proof is essentially the same as given by~\cite{List2016}, but now written for an algebraic problem.

\begin{proof}[Proof of Lemma~\ref{lemma:abstract:l-scheme:convergence}]
 Let $\epV{i}=\pV{i}-\pV{}$. Then taking the difference of Eq.~\eqref{algebraic:l-scheme:abstract} 
 and \eqref{algebraic:nonlinear:abstract} yields
 \begin{align*}
  \LMatrix \left( \epV{i} - \epV{i-1} \right) + \left( \bV(\pV{i-1}) - \bV(\pV{}) \right) 
  + \tau \DM \KM(\pV{i-1}) \DM^{\top} \epV{i} 
  + \tau \DM \left(\KM(\pV{i-1}) - \KM(\pV{})\right)  \left(\rhsVFlux + \DM^{\top} \pV{}\right)
  &= 0.
 \end{align*}
 Multiplying with $\epV{i}$ and applying elementary algebraic manipulations, yields
 \begin{align}
  &\frac{L}{2} \| \epV{i} \|_{\MassMatrixP{}}^2 + \frac{L}{2} \| \epV{i} - \epV{i-1} \|_{\MassMatrixP{}}^2 - \frac{L}{2} \| \epV{i-1} \|_{\MassMatrixP{}}^2	\label{proof:abstract-l-scheme:eq-1:1} \\
  &\quad+ \langle \bV(\pV{i-1}) - \bV(\pV{}) , \epV{i-1} \rangle                                                        	                        \label{proof:abstract-l-scheme:eq-1:2} \\
  &\quad+ \langle \bV(\pV{i-1}) - \bV(\pV{}) , \epV{i} - \epV{i-1} \rangle                                                                        	\label{proof:abstract-l-scheme:eq-1:3} \\
  &\quad+ \tau \langle \KM(\pV{i-1}) \DM^{\top} \epV{i} , \DM^\top \epV{i} \rangle                                                                    	\label{proof:abstract-l-scheme:eq-1:4} \\
  &\quad+ \tau \langle \left(\KM(\pV{i-1}) - \KM(\pV{})\right)  \left(\rhsVFlux + \DM^{\top} \pV{}\right), \DM^\top \epV{i} \rangle                     \label{proof:abstract-l-scheme:eq-1:5}
  =0.  
 \end{align}
 By employing (L1), we obtain for the term~\eqref{proof:abstract-l-scheme:eq-1:2}
 \begin{align} \label{proof:abstract-l-scheme:eq-2}
  \langle \bV(\pV{i-1}) - \bV(\pV{}) , \epV{i-1} \rangle 
  \geq 
  \frac{1}{\LipschitzB} \left\| \bV(\pV{i-1}) - \bV(\pV{}) \right\|_{\MassMatrixP{-1}}^2.
 \end{align}
 By employing the Cauchy-Schwarz inequality and Young's inequality, we obtain for the term~\eqref{proof:abstract-l-scheme:eq-1:3}
 \begin{align}\label{proof:abstract-l-scheme:eq-3}
  \langle \bV(\pV{i-1}) - \bV(\pV{}) , \epV{i} - \epV{i-1} \rangle 
  \geq - \frac{1}{2L}  \left\|\bV(\pV{i-1}) - \bV(\pV{}) \right\|_{\MassMatrixP{-1}}^2  
       - \frac{L}{2} \left\| \epV{i} - \epV{i-1} \right\|_{\MassMatrixP{}}^2. 
 \end{align}
 By employing Assumption~(L2), we obtain for the term~\eqref{proof:abstract-l-scheme:eq-1:4}
 \begin{align}\label{proof:abstract-l-scheme:eq-4}
  \langle \KM(\pV{i-1}) \DM^{\top} \epV{i} , \DM^\top \epV{i} \rangle 
  \geq
  \km \| \DM^\top \epV{i} \|_{\MassMatrixQ{-1}}^2.
 \end{align}
 By employing Cauchy-Schwarz, Young's inequality, Assumption~(L2)--(L3), we obtain for
 the term~\eqref{proof:abstract-l-scheme:eq-1:5}
 \begin{align}\nonumber
  \langle \left(\KM(\pV{i-1}) - \KM(\pV{})\right)  \left(\rhsVFlux + \DM^{\top} \pV{}\right), \DM^\top \epV{i} \rangle
  &\geq
  - \frac{1}{2\km} \| {\MassMatrixQ{-1}} (\rhsVFlux + \DM^{\top} \pV{}) \|_\infty^2 \| (\KM(\pV{i-1}) - \KM(\pV{})) \MassMatrixQ{} \|_{\MassMatrixQ{},\infty}^2 - \frac{\km}{2} \| \DM^\top \epV{i} \|_{\MassMatrixQ{-1}}^2 \\
  &\geq
  - \frac{1}{2\km} q_\infty^2 \LipschitzK^2 \| \bV(\pV{i-1}) - \bV(\pV{}) \|_{\MassMatrixQ{-1}}^2 - \frac{\km}{2} \| \DM^\top \epV{i} \|_{\MassMatrixQ{-1}}^2.  
  \label{proof:abstract-l-scheme:eq-5}
 \end{align}
 Inserting Eq.~\eqref{proof:abstract-l-scheme:eq-2}--\eqref{proof:abstract-l-scheme:eq-5} into 
 Eq.~\eqref{proof:abstract-l-scheme:eq-1:1}--\eqref{proof:abstract-l-scheme:eq-1:5}, yields
 \begin{align}
    \left(\frac{1}{\LipschitzB} - \frac{1}{2L} - \tau \frac{q_\infty^2 \LipschitzK^2}{2\km} \right) \| \bV(\pV{i-1}) - \bV(\pV{}) \|_{\MassMatrixP{-1}}^2
  + \frac{L}{2} \| \epV{i} \|_{\MassMatrixP{}}^2
  + \tau \frac{\km}{2} \| \DM^\top \epV{i} \|_{\MassMatrixQ{-1}}^2                                                                    	
  \leq 
  \frac{L}{2} \| \epV{i-1} \|_{\MassMatrixP{}}^2.
 \end{align}
 Assuming $\frac{1}{\LipschitzB} - \frac{1}{2L} - \tau \frac{q_\infty^2 \LipschitzK^2}{2\km}\geq0$ and 
 applying an algebraic Poincar\'e inequality, yields the final result.
\end{proof}


\bibliographystyle{ieeetr}
\bibliography{l-scheme-biot}

\end{document}